\documentclass{article}
\usepackage[latin9]{inputenc}
\usepackage{mathtools}
\usepackage{amsmath}
\usepackage{amsthm}
\usepackage{amssymb}
\usepackage[unicode=true,
 bookmarks=false,
 breaklinks=false,pdfborder={0 0 1},backref=section,colorlinks=false]
 {hyperref}
\hypersetup{
 hidelinks}

\makeatletter
\theoremstyle{plain}
\newtheorem{thm}{\protect\theoremname}
\theoremstyle{remark}
\newtheorem{rem}[thm]{\protect\remarkname}
\theoremstyle{remark}
\newtheorem*{rem*}{\protect\remarkname}
\theoremstyle{plain}
\newtheorem{prop}[thm]{\protect\propositionname}
\theoremstyle{definition}
\newtheorem{defn}[thm]{\protect\definitionname}

\@ifundefined{date}{}{\date{}}
\usepackage{enumerate}
\usepackage{comment}

\providecommand{\definitionname}{Definition}
\providecommand{\propositionname}{Proposition}
\providecommand{\remarkname}{Remark}
\providecommand{\theoremname}{Theorem}

\makeatother

\providecommand{\definitionname}{Definition}
\providecommand{\propositionname}{Proposition}
\providecommand{\remarkname}{Remark}
\providecommand{\theoremname}{Theorem}

\begin{document}
\title{Variational reduction of homogeneous Lagrangian systems}
\author{J. Fernández$\;^{1}$, S. Grillo$\;^{1,2}$, JC Marrero$\;^{3}$,
E. Padrón$\;^{3}$}

\maketitle
\vspace{-20pt}

\begin{center}
\textit{\small{}{}$^{1}$ Instituto Balseiro, Universidad Nacional
de Cuyo and $^{2}$CONICET, Av. Bustillo 9500, San Carlos de Bariloche
R8402AGP, República Argentina}{\small\par}
\par\end{center}

\begin{center}
 \textit{\small{}{}e-mail: jfernand@ib.edu.ar, sergio.grillo@ib.edu.ar }{\small\par}
\par\end{center}

\begin{center}
 
\par\end{center}

\begin{center}
\textit{\small{}{}$\;^{3}$ULL-CSIC Geometr{í}a Diferencial y Mecánica
Geométrica, Departamento de Matemáticas, Estad{í}stica e Investigación
Operativa and Instituto de Matemáticas y Aplicaciones (IMAULL)}{\small{}}\\
{\small{} }\textit{\small{}{}University of La Laguna, Spain }\\
\textit{\small{} {}e-mail: jcmarrer@ull.edu.es, mepadron@ull.edu.es
}{\small{} }{\small\par}
\par\end{center}

\begingroup 
\global\long\def\thefootnote{}%
 \footnotetext{\hspace{-20pt} AMS Mathematics Subject Classification
(2020): Primary 70H03, 70H30, 70G65; Secundary 53D05, 53D10.

} \endgroup

\vspace{30pt}

\hfill{}{ \textit{{ In memory of our friend H. Cendra}}}
\begin{abstract}
In this paper we show that a variational reduction procedure can be
defined for Lagrangian systems subject to scaling symmetries, i.e.
for homogeneous Lagrangian systems, in such a way that the trajectories
of the system can be reconstructed up to quadratures from the critical
points of the reduced variational principle. Also, we characterize
the mentioned critical points in terms of a set of ordinary differential
equations which are the scaling analogue of the Lagrange-Poincaré
equations. Finally, we study if the homogeneous Lagrangian systems
are naturally related or not with the Herglotz variational principle. 
\end{abstract}
\tableofcontents{}

\section{Introduction}

Homogeneous Hamiltonian systems have been extensively studied in recent
years from very different perspectives (see, for instance, Refs. \cite{Ar,bgmp,BrJaSl,BGG,GG,grab,gralo,Ma}).
One of these perspectives is related to the scaling reduction process.
Here, we say that a Hamiltonian system is \textit{homogeneous}, or
that it has a \textit{scaling symmetry}, if the Hamiltonian function
and the Poisson (or symplectic) structure defining the system are
homogeneous of degree $1$ with respect to some action of the multiplicative
group of real numbers (or its subgroup of positive numbers) on the
phase space of the system.

A symplectic Hamiltonian system with a scaling symmetry generates,
after a reduction process and on the reduced manifold $C,$ a distribution
of co-dimension 1 which is maximally non-integrable; that is, a contact
structure on $C$. The Hamiltonian dynamics projects onto the reduced
space $C$ and the corresponding dynamical vector field can be defined
from the mentioned contact structure. Similar results were obtained
for Poisson homogeneous Hamiltonian systems. In this case, a Kirillov
structure for a line fiber on the reduced space defines the reduced
dynamics (see Refs. \cite{bgmp} and \cite{GG}). Moreover, the reduced
version of all these systems can be defined in such a way that, from
the trajectories of the reduced system, the trajectories of the original
system can be reconstructed up to one quadrature.

The main aim of the present paper is to study what happens on the
Lagrangian side or, more precisely, in the variational formulation
of Lagrangian systems. Specifically, given a homogeneous Lagrangian
system (i.e. one with a Lagrangian function homogeneous of degree
$1$) with its associated Hamilton variational principle, we ask whether
there exists a reduced variational principle such that, from its critical
points, the critical points of the original Hamilton principle can
be reconstructed. We show that such a principle exists, at least when
the homogeneity is given by a principal action $\psi:\mathbb{R}^{+}\times Q\rightarrow Q$,
where $\mathbb{R}^{+}$ is the multiplicative group of positive real
numbers and $Q$ is the configuration space of the Lagrangian system.
Moreover, as in the Hamiltonian case, the trajectories of the original
Lagrangian system can be reconstructed up to a single quadrature.

The reduced variational principle in question is not defined on the
quotient manifold $Q/\mathbb{R}^{+}$, but rather on the associated
line bundle $\left(Q\times\mathbb{R}\right)/\mathbb{R}^{+}$ related
to $\psi$. This is exactly what happens in variational reduction
by standard symmetries (see \cite{cmr}): if the symmetry is given
by a Lie group $G$, a reduced variational principle can be defined
on the associated bundle $\tilde{\mathfrak{g}}=\left(Q\times\mathfrak{g}\right)/G$,
where $\mathfrak{g}$ is the Lie algebra of $G$, in such a way that
the critical points of the Hamilton principle can be reconstructed
from those of the reduced principle. Thus, in some sense, we extend
to scaling symmetries a classical result on variational reduction
valid for standard symmetries.

The paper is organized as follows. In Section \ref{prel} we recall
some basic definitions and results on variational principles, affine
connections, and principal actions. In Section \ref{lsss} we present
a brief review of variational reduction by standard symmetries as
described in Ref. \cite{cmr}, focusing on the case in which the symmetry
group is abelian and the associated principal bundle has a flat connection.
The reconstruction process is also described in this particular case.
In the main section of the paper, Section \ref{hls}, we define what
we mean by a (positive) homogeneous Lagrangian system, describe several
examples, and construct the reduced variational principle mentioned
above together with the corresponding reconstruction process. We also
show that the critical points of that principle are given by the solutions
of a system of ordinary differential equations similar to the so-called
Lagrange-Poincaré equations (see again Ref. \cite{cmr}). In Section
\ref{S5}, we study the relationship between our reduced variational
principle and the Herglotz variational principle \cite{gg,her} (see
also \cite{lainz} and references therein). We show there exist homogeneous
Lagrangian systems such that, for each one of them, the set of critical
points of its associated reduced variational principle does not coincide
with the set of critical points of the Herglotz variational principle
associated to any action-dependent Lagrangian. Conversely, we show
there exist action-dependent Lagrangians such that, for each on of
them, the set of critical points of its associated Herglotz variational
principle does not coincide with the set of critical points of the
reduced variational principle related to any homogeneous Lagrangian
system. The paper ends with the description of future research directions
in Section \ref{S6} and with an appendix which contains the proof
of a technical result which was used in Section \ref{S5.2}.

\bigskip{}

Hernán Cendra, our esteemed colleague and dear friend, was always
interested in Reduction Theory in different frameworks (see, for instance,
Refs. \cite{cg,CMPR,cmr}). We believe that he would have appreciated
our approach in this paper, which we dedicate to his memory.

\bigskip{}

We assume that the reader is familiar with the basic notions of differential
geometry, Lie groups, and Hamiltonian and Lagrangian mechanics (see
for instance Ref. \cite{AbMa}). We shall work in the $C^{\infty}$-category,
and all the manifolds will be finite-dimensional. Also, we shall denote
the canonical tangent and cotangent projections on a manifold $X$
by $\tau_{X}:TX\rightarrow X$ and $\pi_{X}:T^{*}X\rightarrow X$,
respectively. 

\section{Some preliminary definitions and conventions}

\label{prel} In this section, we will review some definitions and
basic constructions on variational principles, affine connections
on vector bundles and principal actions and connections (for more
details see, for instance, \cite{AbMa,GHV})

\subsection{Variational principles}

\label{varpri} Given $\tau>0$, a manifold $M$ and a curve $c:\left[0,\tau\right]\rightarrow M$,
by an \textbf{infinitesimal variation} (or simply a \textbf{variation})
of $c$ we shall mean a curve $\delta c:\left[0,\tau\right]\rightarrow TM$
such that 
\[
\delta c\left(t\right)\in T_{c\left(t\right)}M,\quad\forall t\in\left[0,\tau\right].
\]
Clearly, it can be described by the formula 
\begin{equation}
\delta c\left(t\right)=\left.\frac{\partial}{\partial s}\right|_{s=0}\Gamma\left(s,t\right),\label{dcg}
\end{equation}
for some $\epsilon>0$ and some smooth function $\Gamma:\left(-\epsilon,\epsilon\right)\times\left[0,\tau\right]\rightarrow M$
such that 
\[
\Gamma\left(0,t\right)=c\left(t\right),\quad\forall t\in\left[0,\tau\right].
\]
And given an $\mathbb{R}$-valued functional $\mathfrak{F}$ on the
set of all smooth curves $c:\left[0,\tau\right]\rightarrow M$, by
the \textbf{variation of} $\mathfrak{F}$ at $c$, with respect to
an infinitesimal variation $\delta c$, we mean the quantity 
\[
\delta\mathfrak{F}\left(c\right)\coloneqq\left.\frac{\partial}{\partial s}\right|_{s=0}\mathfrak{F}\left(\Gamma\left(s,\cdot\right)\right).
\]
(We shall assume $\mathfrak{F}$ is regular enough so that $\delta\mathfrak{F}\left(c\right)$
does not depend on the particular function $\Gamma$ that defines
$\delta c$). 
\begin{rem}
\label{invar} For every expression that depends on the values of
a curve $c$ and its derivatives, say $f\left(c\left(t\right),\dot{c}\left(t\right),\ldots\right)$,
we shall write 
\[
\delta\left(f\left(c\left(t\right),\dot{c}\left(t\right),\ldots\right)\right)\coloneqq\left.\frac{\partial}{\partial s}\right|_{s=0}f\left(\Gamma\left(s,t\right),\frac{\partial}{\partial t}\Gamma\left(s,t\right),\ldots\right).\quad\diamond
\]
\end{rem}

Finally, we shall say that a curve $c$ is a \textbf{critical point}
of $\mathfrak{F}$, \textbf{with respect to a given subset} of variations
$\mathcal{V}$, if 
\[
\delta\mathfrak{F}\left(c\right)=0,\quad\forall\delta c\in\mathcal{V};
\]
and we shall say that the pair $\left(\mathfrak{F},\mathcal{V}\right)$
is a \textbf{variational principle} (on $M$). When the subset $\mathcal{V}$
of variations is clear from the context, we shall just say that $c$
is a \textit{critical point of} $\mathfrak{F}$. 

\subsection{Affine connections}

\label{affcon} Consider a vector bundle $\Pi:\mathsf{E}\rightarrow M$
and an affine connection $\nabla$ on $\mathsf{E}$. This gives rise
to a diffeomorphism 
\begin{equation}
\Omega:T\mathsf{E}\rightarrow\mathsf{E}\oplus TM\oplus\mathsf{E}\label{Om}
\end{equation}
such that 
\[
\Omega\left(V\right)=\tau_{\mathsf{E}}\left(V\right)\oplus\Pi_{*}\left(V\right)\oplus\frac{D}{Dt}w\left(0\right),
\]
where $\oplus$ denotes the Whitney sum for bundles on $M$, $D/Dt$
is the covariant derivative associated to $\nabla$ and $w:\left(-\epsilon,\epsilon\right)\rightarrow\mathsf{E}$
is a curve such that 
\[
w\left(0\right)=\tau_{\mathsf{E}}\left(V\right)\quad\textrm{and}\quad\frac{d}{dt}w\left(0\right)=V.
\]
The inverse of $\Omega$ is given by 
\[
\Omega^{-1}\left(X\oplus Y\oplus Z\right)=\frac{d}{dt}\hat{w}\left(0\right),
\]
where $\hat{w}:\left(-\epsilon,\epsilon\right)\rightarrow\mathsf{E}$
satisfies 
\[
\hat{w}\left(0\right)=X,\quad\Pi_{*}\left(\frac{d}{dt}\hat{w}\left(0\right)\right)=Y\quad\textrm{and}\quad\frac{D}{Dt}\hat{w}\left(0\right)=Z.
\]
Also, for each $v\in\mathsf{E}$, we have the linear isomorphism 
\[
\Omega_{v}:T_{v}\mathsf{E}\rightarrow T_{\Pi\left(v\right)}M\oplus\mathsf{E}_{\Pi\left(v\right)}
\]
given by 
\[
\Omega\left(V\right)=v\oplus\Omega_{v}\left(V\right).
\]
Here $\mathsf{E}_{\Pi\left(v\right)}$ is the fiber of $\Pi:E\to M$
over $\Pi(v)$, that is, $\mathsf{E}_{\Pi\left(v\right)}=\Pi^{-1}(\Pi(v)).$
\bigskip{}

If $\mathsf{E}=TM$ and $\Pi=\tau_{M}$, we have the diffeomorphism
\[
\Omega:TTM\rightarrow TM\oplus TM\oplus TM.
\]
And given a curve $c:\left(-\epsilon,\epsilon\right)\rightarrow M$
and a variation $\delta c:\left(-\epsilon,\epsilon\right)\rightarrow TM$
of it, it follows that 
\begin{equation}
\Omega_{\delta c\left(t\right)}\left(\frac{d}{dt}\delta c\left(t\right)\right)=\dot{c}\left(t\right)\oplus\frac{D}{Dt}\delta c\left(t\right).\label{Ov}
\end{equation}
\bigskip{}

Let us finally say that, for each function $f:\mathsf{E}\rightarrow\mathbb{R}$,
we can define the \textbf{fiber derivative }(or \textbf{Legendre transform})
$\mathbb{F}f:\mathsf{E}\rightarrow\mathsf{E}^{*}$ and the \textbf{base
derivative} $\mathbb{B}f:\mathsf{E}\rightarrow T^{*}M$ by the equation
\begin{equation}
\left(\Omega_{v}^{*}\right)^{-1}\left(df\left(v\right)\right)=\mathbb{B}f\left(v\right)\oplus\mathbb{F}f\left(v\right),\label{fbd}
\end{equation}
where $\Omega_{v}^{*}$ is the dual of $\Omega_{v}$. Of course, $\mathbb{F}f$
does not depend on the connection, and it is also given by 
\[
\left\langle \mathbb{F}f\left(v\right),w\right\rangle =\left.\frac{d}{dt}\right|_{t=0}f\left(v+tw\right).
\]

\subsection{Principal actions and the Atiyah diffeomorphism}

\label{paad} Given a manifold $Q$ and a Lie group $G$, suppose
we have an action $\psi:G\times Q\rightarrow Q$ such that the canonical
projection $\pi:Q\rightarrow\left.Q\right/G$ is a principal bundle.
We shall say that $\psi$ is a \textbf{principal action}. Let $\mathfrak{g}$
be the Lie algebra of $G$. Related to the action $\psi$ we have
the \textbf{adjoint bundle} $\tilde{\mathfrak{g}}\rightarrow\left.Q\right/G$,
where 
\[
\tilde{\mathfrak{g}}=\frac{Q\times\mathfrak{g}}{G},
\]
and the action involved in the last quotient is given by 
\[
\left(g,\left(q,\eta\right)\right)\in G\times\left(Q\times\mathfrak{g}\right)\mapsto\left(\psi_{g}\left(q\right),\mathsf{Ad}_{g}\eta\right)\in Q\times\mathfrak{g}.
\]
If $G$ is an abelian group, then $\mathsf{Ad}_{g}$ is the identity
and we have the canonical identification 
\begin{equation}
\tilde{\mathfrak{g}}\xrightarrow{\thicksim}\left(\left.Q\right/G\right)\times\mathfrak{g}:\left[\left(q,\eta\right)\right]\mapsto\left(\left[q\right],\eta\right).\label{nid}
\end{equation}

Fix a principal connection form $\varpi:TQ\rightarrow\mathfrak{g}$
for $\pi$ and denote $p:TQ\rightarrow\left.TQ\right/G$ the canonical
projection related to the tangent lift of $\psi$. The map 
\[
\alpha_{\varpi}:\left.TQ\right/G\rightarrow T\left(\left.Q\right/G\right)\oplus\tilde{\mathfrak{g}}
\]
given by 
\[
\alpha_{\varpi}\left(p\left(v_{q}\right)\right)=\pi_{*}\left(v_{q}\right)\oplus\left[\left(q,\varpi\left(v_{q}\right)\right)\right]
\]
is a diffeomorphism: the \textbf{Atiyah diffeomorphism}. When $G$
is abelian, according to the identification \eqref{nid}, the latter
can be seen as a map 
\begin{equation}
\alpha_{\varpi}:\left.TQ\right/G\rightarrow T\left(\left.Q\right/G\right)\times\mathfrak{g}\label{adiff}
\end{equation}
such that 
\begin{equation}
\alpha_{\varpi}\left(p\left(v_{q}\right)\right)=\left(\pi_{*}\left(v_{q}\right),\varpi\left(v_{q}\right)\right).\label{adiff2}
\end{equation}

\section{Lagrangian systems with standard symmetries}

\label{lsss} In this section we briefly review the variational reduction
and reconstruction of the Lagrangian systems subject to standard symmetries
(in the abelian case).

\subsection{Variational formulation and the Euler-Lagrange equations}

Consider a Lagrangian system on a manifold $Q$ with Lagrangian function
$L:TQ\rightarrow\mathbb{R}$. We shall indicate it by the pair $\left(Q,L\right)$.
Recall that, given $\tau>0$, its trajectories $\gamma:\left[0,\tau\right]\rightarrow Q$
are (by definition) the critical points of the functional 
\begin{equation}
A\left(\gamma\right)=\int_{0}^{\tau}L\left(\dot{\gamma}\left(t\right)\right)\,\mathsf{d}t,\label{aL}
\end{equation}
the \textbf{action}\textit{ }of $L$, w.r.t. infinitesimal variations
$\delta\gamma:\left[0,\tau\right]\rightarrow TQ$ satisfying 
\begin{equation}
\delta\gamma\left(0\right),\delta\gamma\left(\tau\right)\in\mathsf{O}_{Q},\label{eLv}
\end{equation}
where $\mathsf{O}_{Q}\subseteq TQ$ is the null distribution on $Q$
(i.e. the image of the zero section of $\tau_{Q}$). Eqs. \eqref{aL}
and \eqref{eLv} define the so-called \textbf{Hamilton Principle}.
As shown in Ref. \cite{cg}, and following the notation and terminology
of Section \ref{affcon}, if we fix a torsion-free affine connection
for $\tau_{Q}:TQ\rightarrow Q$, the critical points of above variational
principle can be described as the solutions $\gamma:\left[0,\tau\right]\rightarrow Q$
of the equations 
\begin{equation}
-\frac{D}{Dt}\mathbb{F}L\left(\dot{\gamma}\left(t\right)\right)+\mathbb{B}L\left(\dot{\gamma}\left(t\right)\right)=0,\quad\forall t\in\left(0,\tau\right),\label{9'}
\end{equation}
where the fiber and base derivatives $\mathbb{F}L:TQ\rightarrow T^{*}Q$
and $\mathbb{B}L:TQ\rightarrow T^{*}Q$ are given as in Eq. \eqref{fbd}.
The latter gives a global expression of the \textbf{Euler-Lagrange
equations}. In fact, if $(q^{i})$ are local coordinates on an open
subset $U$ of $Q$, $(q^{i},\dot{q}^{i})$ are the corresponding
fibered local coordinates on $\tau_{Q}^{-1}(U)$ and we take the local
flat affine connection for $(\tau_{Q})_{|\tau_{Q}^{-1}(U)}:\tau_{Q}^{-1}(U)\to U$,
then 
\[
\mathbb{F}L(q^{i},\dot{q}^{i})=(q^{i},\frac{\partial L}{\partial\dot{q}^{i}}),\;\;\;\;\mathbb{B}L(q^{i},\dot{q}^{i})=(q^{i},\frac{\partial L}{\partial{q}^{i}}),\;\;\;\;\frac{D}{Dt}(\mathbb{F}L(q^{i},\dot{q}^{i}))=(q^{i},\frac{d}{dt}(\frac{\partial L}{\partial\dot{q}^{i}}))
\]
and Eq. \eqref{9'} are the standard Euler-Lagrange equations 
\[
\frac{d}{dt}\left(\frac{\partial L}{\partial q^{i}}\right)-\frac{\partial L}{\partial q^{i}}=0,\;\;\mbox{for all \ensuremath{i}}.
\]

\subsection{The reduced variational principle and the reconstruction equations}

Suppose we have a principal action $\psi:G\times Q\rightarrow Q$
of some Lie group $G$, with related principal bundle $\pi:Q\rightarrow\left.Q\right/G$,
such that 
\[
L\left(\left(\psi_{g}\right)_{*,q}\left(v\right)\right)=L\left(v\right),\quad\forall g\in G,\;q\in Q,\;v\in T_{q}Q,
\]
i.e. $L$ is $G$-invariant w.r.t. the tangent lift of $\psi$. In
other words, $\left(Q,L\right)$ has a \textit{standard symmetry}
defined by the Lie group $G$. For later convenience, we shall assume
that $G$ is abelian. Following the procedure developed in Ref. \cite{cmr},
and using the notation of Section \ref{paad}, fix a principal connection
form $\varpi:TQ\rightarrow\mathfrak{g}$ for $\pi$ and consider the
related Atiyah diffeomorphism $\alpha_{\varpi}$ (see Eqs. \eqref{adiff}
and \eqref{adiff2}). Note that, since $\tilde{\mathfrak{g}}$ can
be identified with the trivial bundle $\left(\left.Q\right/G\right)\times\mathfrak{g}$
(recall Eq. \eqref{nid}), 
\begin{equation}
T\tilde{\mathfrak{g}}\simeq\mathfrak{g}\times T\left(\left.Q\right/G\right)\times\mathfrak{g}.\label{tgt}
\end{equation}
For simplicity, assume that $\varpi$ can be chosen flat. With all
that, we can define: 
\begin{itemize}
\item the \textbf{reduced Lagrangian} 
\[
\ell:T\left(\left.Q\right/G\right)\times\mathfrak{g}\rightarrow\mathbb{R},
\]
given by the equation $L=\ell\circ\alpha_{\varpi}\circ p$, where
$p:TQ\rightarrow\left.TQ\right/G$ is the canonical projection related
to the tangent lift of $\psi$, and the \textbf{reduced variational
principle}\textit{, }given by 
\item the\textit{ }\textbf{reduced action}\textit{ } 
\begin{equation}
\hat{A}\left(x,y\right)=\int_{0}^{\tau}\ell\left(\dot{x}\left(t\right),y\left(t\right)\right)\,\mathsf{d}t,\label{raL}
\end{equation}
defined on curves $\left\{ \left(x,y\right):\left[0,\tau\right]\rightarrow\left(\left.Q\right/G\right)\times\mathfrak{g}\simeq\tilde{\mathfrak{g}}\right\} $, 
\item and the \textbf{reduced infinitesimal variations}\textit{ }(see Eq.
\eqref{tgt})\textit{ 
\[
\delta\left(x,y\right)\simeq\left(y,\delta x,\delta y\right):\left[0,\tau\right]\rightarrow T\tilde{\mathfrak{g}}\simeq\mathfrak{g}\times T\left(\left.Q\right/G\right)\times\mathfrak{g},
\]
} where $\delta y:\left[0,\tau\right]\rightarrow\mathfrak{g}$ is
of the form 
\begin{equation}
\delta y\left(t\right)=\frac{d}{dt}\eta\left(t\right),\quad\textrm{for some}\quad\eta:\left[0,\tau\right]\rightarrow\mathfrak{g},\label{reLv1}
\end{equation}
and $\delta x$ and $\eta$ satisfy the boundary conditions 
\begin{equation}
\delta x\left(0\right),\delta x\left(\tau\right)\in\mathsf{O}_{\left.Q\right/G},\quad\eta\left(0\right)=\eta\left(\tau\right)=0.\label{reLv}
\end{equation}
\end{itemize}
\begin{rem*}
Note that the conditions \eqref{reLv1} and \eqref{reLv} for $\delta y$
can be condensed in just one condition: 
\begin{equation}
\int_{0}^{\tau}\delta y\left(s\right)\,ds=0.\quad\diamond\label{oc}
\end{equation}
\end{rem*}
It can be shown that, if $\gamma:\left[0,\tau\right]\rightarrow Q$
is a trajectory of $\left(Q,L\right)$, then 
\[
\left(x\left(t\right),y\left(t\right)\right)\coloneqq\left(\pi\left(\gamma\left(t\right)\right),\varpi\left(\dot{\gamma}\left(t\right)\right)\right)
\]
is a critical point of the reduced variational principle (see Eqs.
\eqref{raL}, \eqref{reLv1} and \eqref{reLv}). Conversely, given
a curve $\left(x,y\right)$ satisfying such a principle, and fixing
a $\varpi$-horizontal lift $h:\left[0,\tau\right]\rightarrow Q$
of $x$, then 
\[
\gamma\left(t\right)\coloneqq\psi\left(g\left(t\right),h\left(t\right)\right)
\]
is a trajectory of $\left(Q,L\right)$ if and only if $g\left(t\right)$
satisfies the \textbf{reconstruction equation}: 
\[
\dot{g}\left(t\right)=\left(L_{g\left(t\right)}\right)_{*,e}\left(y\left(t\right)\right).
\]
Since we are assuming that $G$ is abelian, we can solve this equation
up to quadratures (see Ref. \cite{mmr}). The general solution is
\[
g\left(t\right)=g_{0}\cdot\exp\left(\int_{0}^{t}y\left(s\right)\,\mathsf{d}s\right),
\]
where $g_{0}\in G$ and $\exp$ is the exponential map of $G$. As
a consequence, 
\begin{equation}
\gamma\left(t\right)=\psi\left(g_{0}\cdot e^{\int_{0}^{t}y\left(s\right)\,\mathsf{d}s},h\left(t\right)\right)=\psi\left(e^{\int_{0}^{t}y\left(s\right)\,\mathsf{d}s},\psi\left(g_{0},h\left(t\right)\right)\right).\label{gtif}
\end{equation}

\begin{rem*}
Moreover, given $q_{0}\in Q$ and $v_{0}\in T_{q_{0}}Q$, if there
exists a critical point $\left(x,y\right)$ such that $x\left(0\right)=\pi\left(q_{0}\right)$,
$\dot{x}\left(0\right)=\pi_{*}\left(v_{0}\right)$ and $y\left(0\right)=\varpi\left(v_{0}\right)$,
then we can choose $h$ and $g_{0}$ such that $\dot{\gamma}\left(0\right)=v_{0}$
(\textit{ipso facto} $\gamma\left(0\right)=q_{0}$). $\quad\diamond$ 
\end{rem*}
\bigskip{}

Summing up, given a Lagrangian system with a standard (abelian) symmetry,
there exists a reduced variational principle and a way to reconstruct
up to quadratures, from its critical points, the critical points of
the original variational principle (the Hamilton principle). In the
rest of the paper, we want to study if something similar can be done
in the context of \textit{scaling symmetries}. Before doing that,
let us see how the critical points of $\hat{A}$ can be described
in terms of ordinary differential equations.

\subsection{The Lagrange-Poincaré equations}

\label{LPe}

Consider an affine connection $\nabla$ on $T\left(Q/G\right)$, the
related diffeomorphism 
\[
\Omega:TT\left(Q/G\right)\rightarrow T\left(Q/G\right)\oplus T\left(Q/G\right)\oplus T\left(Q/G\right),
\]
and the linear isomorphisms $\Omega_{\dot{x}}$, for all $\dot{x}\in T\left(Q/G\right)$
(see Section \ref{affcon}). For the reduced Lagrangian $\ell:T\left(\left.Q\right/G\right)\times\mathfrak{g}\rightarrow\mathbb{R}$,
let us write $d_{1}\ell$ to indicate the differential with respect
to the factor $T\left(Q/G\right)$, i.e. 
\begin{equation}
\left\langle d\ell\left(\dot{x},y\right),\left(v,0\right)\right\rangle =\left\langle d_{1}\ell\left(\dot{x},y\right),v\right\rangle ,\quad\forall v\in T_{\dot{x}}T\left(Q/G\right),\label{dfb1}
\end{equation}
for all $y\in\mathfrak{g}$, and define 
\begin{equation}
\frac{\partial\ell}{\partial x}:T\left(Q/G\right)\times\mathfrak{g}\rightarrow T^{*}\left(Q/G\right)\label{dfb2}
\end{equation}
and 
\begin{equation}
\frac{\partial\ell}{\partial\dot{x}}:T\left(Q/G\right)\times\mathfrak{g}\rightarrow T^{*}\left(Q/G\right)\label{dfb3}
\end{equation}
through the equation 
\begin{equation}
\left(\Omega_{\dot{x}}^{*}\right)^{-1}\left(d_{1}\ell\left(\dot{x},y\right)\right)=\frac{\partial\ell}{\partial x}\left(\dot{x},y\right)\oplus\frac{\partial\ell}{\partial\dot{x}}\left(\dot{x},y\right).\label{dfb4}
\end{equation}
(These objects are similar to the base and fiber derivatives defined
in Eq. \eqref{fbd}, but with an additional \textit{parameter} $y$).
In these terms, it can be shown (see Ref. \cite{cmr}) that the pair
of curves $\left(x,y\right)$ is a critical point of the reduced variational
principle if and only if $\left(x,y\right)$ satisfies the \textbf{horizontal
Lagrange-Poincaré equation}\footnote{\textit{Recall we are assuming that the chosen principal connection
$\varpi$ is flat. This is why the curvature appearing in Ref. }\cite{cmr}\textit{
does not appear here.}} 
\begin{equation}
-\frac{D}{Dt}\frac{\partial\ell}{\partial\dot{x}}\left(\dot{x}\left(t\right),y\left(t\right)\right)+\frac{\partial\ell}{\partial x}\left(\dot{x}\left(t\right),y\left(t\right)\right)=0\label{ehor}
\end{equation}
and the \textbf{vertical Lagrange-Poincaré equation}\footnote{\textit{Recall also that we are assuming that $G$ is abelian.}}
\begin{equation}
-\frac{d}{dt}\frac{\partial\ell}{\partial y}\left(\dot{x}\left(t\right),y\left(t\right)\right)=0.\label{ever}
\end{equation}

Here $D/Dt$ denotes the covariant derivative of curves on $T^{*}\left(Q/G\right)$
related to the affine connection $\nabla$.\bigskip{}

At the end of the paper, we shall see that a set of equations with
a similar structure can be written in the context of scaling symmetries.

Note that if $(x^{i})$ are local coordinates on an open subset $\widehat{U}$
of $Q/G$, $(x^{i},\dot{x}^{i})$ are the corresponding local coordinates
on $\tau_{Q/G}^{-1}(\widehat{U}),$ $(y^{\alpha})$ are linear coordinates
on the flat Lie algebra ${\mathfrak{g}}$ induced by a basis of ${\mathfrak{g}}$,
and we take the local flat affine connection on $(\tau_{Q/G})_{|\tau_{Q}^{-1}(\widehat{U})}:\tau_{Q}^{-1}(\widehat{U})\to\widehat{U}$,
then Eqs. \eqref{ehor} and \eqref{ever} may be described as 
\[
\frac{d}{dt}\left(\frac{\partial\ell}{\partial\dot{x}}\right)-\frac{\partial\ell}{\partial x}=0,\;\;\;\frac{d}{dt}\left(\frac{\partial\ell}{\partial y}\right)=0.
\]

\section{Lagrangian systems with scaling symmetries}

\label{hls}

Given a Lagrangian system $\left(Q,L\right)$, if instead of having
a standard symmetry $\left(Q,L\right)$ has a scaling one, then the
Lagrangian function $L$ is not invariant, but homogeneous. In this
case, the following question arises: can we define some kind of \textit{reduced}
variational principle such that, from the critical points of this
principle, the trajectories of $\left(Q,L\right)$ can be reconstructed?
In the present section we give an affirmative answer to this question.
Moreover, we show that, as in the Hamiltonian side, the reconstruction
can be done up to one quadrature.

\subsection{Definitions and basic properties}

Given a manifold $Q$ and a Lagrangian function\footnote{We are not asking $L$ to be regular or hyperregular. So, the singular
Lagrangians are included in our study.} $L:TQ\rightarrow\mathbb{R}$, assume we have a principal action $\psi:\mathbb{R}^{+}\times Q\rightarrow Q$
(see Section \ref{paad}) of the multiplicative group of positive
real numbers $\mathbb{R}^{+}$, such that $L$ is \textbf{homogeneous
of degree 1} (or simply \textbf{homogeneous}, from now on) w.r.t.
the tangent lift of $\psi$, i.e. 
\[
L\left(\left(\psi_{\varsigma}\right)_{*,q}\left(v\right)\right)=\varsigma\,L\left(v\right),\quad\forall\varsigma\in\mathbb{R}^{+},\;q\in Q,\;v\in T_{q}Q.
\]

\begin{rem}
\label{posh} Note that we are actually dealing with \textit{positive
homogeneous }Lagrangians, in the sense that we are considering an
action of the positive reals $\mathbb{R}^{+}$, instead of an action
of the entire multiplicative group of real numbers $\mathbb{R}^{\times}\coloneqq\mathbb{R}-\left\{ 0\right\} $.
$\quad\diamond$ 
\end{rem}

Using that ${\mathbb{R}}^{+}$ is contractible, it follows that the
principal bundle $\pi:Q\to Q/{\mathbb{R}}^{+}$ is trivial (see, for
instance, \cite{St}), which implies the existence of a \textbf{scaling
function}: a positive homogeneous function w.r.t. $\psi$, i.e. a
function $f:Q\rightarrow\mathbb{R}^{+}$ such that 
\begin{equation}
f\left(\psi_{\varsigma}\left(q\right)\right)=\varsigma\,f\left(q\right),\quad\forall\varsigma\in\mathbb{R}^{+},\;q\in Q.\label{sff}
\end{equation}

\begin{rem*}
Actually, we have a scaling function for each global trivialization
$\Phi:Q\rightarrow\left(\left.Q\right/\mathbb{R}^{+}\right)\times\mathbb{R}^{+}$,
and it is given by the second component of the latter. And reciprocally,
given a scaling function $f$, the map $\left(\pi,f\right)$ defines
a global trivialization. For later convenience, let us mention that
the inverse of $\left(\pi,f\right)$, 
\[
\left(\pi,f\right)^{-1}:\left(\left.Q\right/\mathbb{R}^{+}\right)\times\mathbb{R}^{+}\rightarrow Q,
\]
is given by the formula 
\begin{equation}
\left(\pi,f\right)^{-1}\left(\pi\left(q\right),\varsigma\right)=\psi\left(\frac{\varsigma}{f\left(q\right)},q\right).\quad\diamond\label{invt}
\end{equation}
\end{rem*}
In terms of the infinitesimal generator $\triangle$ of $\psi$, the
Eq. \eqref{sff} takes the form 
\[
\mathcal{L}_{\triangle}f=f,
\]
where $\mathcal{L}$ denotes the Lie derivative.

\bigskip{}

Given a scaling function $f$, we have the flat principal connection
$\varpi_{f}\coloneqq\mathsf{d}\ln f:TQ\rightarrow\mathbb{R}$ and
the related Atiyah diffeomorphism 
\[
\alpha_{f}:\left.TQ\right/\mathbb{R}^{+}\rightarrow T\left(\left.Q\right/\mathbb{R}^{+}\right)\times\mathbb{R},
\]
given by (see Eq. \eqref{adiff2}) 
\begin{equation}
\alpha_{f}\left(p\left(v_{q}\right)\right)=\left(\pi_{*}\left(v_{q}\right),\left(\mathsf{d}\ln f\left(v_{q}\right)\right)\right)=\left(\pi_{*}\left(v_{q}\right),\left(\frac{\mathsf{d}f\left(v_{q}\right)}{f\left(q\right)}\right)\right).\label{alfaf}
\end{equation}

\begin{rem*}
Given a curve $\gamma:\left[0,\tau\right]\rightarrow Q$, if we define
\begin{equation}
x\left(t\right)=\pi\left(\gamma\left(t\right)\right)\quad\textrm{and}\quad y\left(t\right)=\mathsf{d}\ln f\left(\dot{\gamma}\left(t\right)\right)=\frac{d\ln f\left(\gamma\left(t\right)\right)}{dt},\label{ito1}
\end{equation}
then 
\begin{equation}
\alpha_{f}\left(p\left(\dot{\gamma}\left(t\right)\right)\right)=\left(\dot{x}\left(t\right),y\left(t\right)\right).\quad\diamond\label{ito2}
\end{equation}
\end{rem*}

\subsection{Some examples}

\subsubsection{Kinetic Lagrangians}

Let $g$ be a Riemannian metric on $Q$ for which there exists a principal
action $\psi:\mathbb{R}^{+}\times Q\rightarrow Q$ such that, for
each $s\in\mathbb{R}^{+}$, the diffeomorphism $\psi_{s}:Q\rightarrow Q$
is a homothety with scale factor $s^{2}$, i.e. 
\[
g(\psi_{s}(q))\left(\left(\psi_{s}\right)_{*,q}\left(u_{q}\right),\left(\psi_{s}\right)_{*,q}\left(v_{q}\right)\right)=s^{2}\,g(q)\left(u_{q},v_{q}\right),
\]
for all $q\in Q$ and $u_{q},v_{q}\in T_{q}Q$. 
\begin{rem*}
A homothety $F:Q\to Q$ is a diffeomorphism such that 
\[
g(F(q))(F_{*,q}(u_{q}),F_{*,q}(v_{q}))=c_{F}\,g(q)(u_{q},v_{q}),\quad\forall q\in Q,\quad u_{q},v_{q}\in T_{q}Q,
\]
for some $c_{F}>0$. Denote $\mathcal{H}\left(Q,g\right)$ the group
of homotheties for $g$. Then, the map $\mathcal{H}\left(Q,g\right)\rightarrow\mathbb{R}^{+}$,
which sends every homothety to its scale factor, is a homomorphism.
It can be shown that the existence of the above mentioned action $\psi$
is equivalent to the existence of a right inverse of such a homomorphism.
In particular, if $\mathcal{I}\left(Q,g\right)\subseteq\mathcal{H}\left(Q,g\right)$
is the subgroup of the isometries for $g$ and $\mathcal{I}\left(Q,g\right)\rightarrow\mathcal{H}\left(Q,g\right)$
is the inclusion map, then the sequence 
\[
\left\{ id_{Q}\right\} \rightarrow\mathcal{I}\left(Q,g\right)\rightarrow\mathcal{H}\left(Q,g\right)\rightarrow\mathbb{R}^{+}
\]
is a short exact sequence (non-necessarily splitting). For instance,
for $Q=\mathbb{R}^{n}-\left\{ 0\right\} $ and $g$ a constant metric,
the multiplication by scalars defines a principal action $\psi$ as
described before. In this case, the previous exact sequence splits,
i.e. 
\[
\mathcal{H}\left(Q,g\right)\simeq\mathbb{R}^{+}\oplus\mathcal{I}\left(Q,g\right).\quad\diamond
\]
\end{rem*}
Now, consider the kinetic Lagrangian function $L=K_{g}:TQ\to{\mathbb{R}}$
given by 
\[
K_{g}(v)=g(v,v),\mbox{ for }v\in TQ,
\]
and the action $\widetilde{\psi}:{\mathbb{R}}^{+}\times Q\to Q$ defined
by 
\[
\widetilde{\psi}(r,q)=\psi(\sqrt{r},q),\mbox{ for }(r,q)\in{\mathbb{R}}^{+}\times Q.
\]
Then, $L$ is clearly homogeneous and, as it is well-known, the trajectories
of $(Q,L)$ are the geodesics of $g$. A natural scaling function
related to $\widetilde{\psi}$ is given by the formula 
\[
f\left(q\right)=g(q)\left(\triangle_{q},\triangle_{q}\right),
\]
where $\triangle$ is the infinitesimal generator for $\widetilde{\psi}$.

\subsubsection{Jacobi metrics}

\label{jm} For the same metric $g$ and action $\widetilde{\psi}$
as above, consider a function $V:Q\rightarrow\mathbb{R}$ such that
\[
V\left(\widetilde{\psi}_{s}\left(q\right)\right)=V\left(q\right),\quad\forall q\in Q,
\]
i.e. $V$ is invariant w.r.t. $\widetilde{\psi}$. Assume also that
$e\in\mathbb{R}^{+}$ is a strict upper bound for $V$. Then, the
Lagrangian $L={\left(e-V\circ\tau_{Q}\right)\,K_{g}}$ is homogeneous
w.r.t. the tangent lift of $\widetilde{\psi}$, and its trajectories
of energy $1$ are reparametrizations of the trajectories of energy
$e$ of the mechanical Lagrangian system $\left(Q,K_{g}-V\circ\tau_{Q}\right)$
(see, for instance, \cite{AbMa}).

\bigskip{}

A concrete example is given by $Q=\mathbb{R}^{+}\times\mathbb{R}^{+}$,
\begin{equation}
\widetilde{\psi}:\left(s,\left(a,b\right)\right)\mapsto\left(\sqrt{s}\,a,\sqrt{s}\,b\right),\label{sab}
\end{equation}
$g$ any constant metric on $\mathbb{R}^{2}$ and $V:\left(a,b\right)\in Q\mapsto v\left(a/b\right)$,
for any bounded function $v:\mathbb{R}^{+}\rightarrow\mathbb{R}$.

\subsubsection{Harmonic oscillator}

The Lagrangian function of an $n$-dimensional (isotropic) harmonic
oscillator is 
\[
L\left(\mathbf{q},\dot{\mathbf{q}}\right)=\frac{1}{2}\,M\,\left\Vert \dot{\mathbf{q}}\right\Vert ^{2}-\frac{1}{2}\,k\,\left\Vert \mathbf{q}\right\Vert ^{2},\quad\forall\mathbf{q},\dot{\mathbf{q}}\in\mathbb{R}^{n},
\]
where $\left\Vert \cdot\right\Vert $ denotes the Euclidean norm on
$\mathbb{R}^{n}$, and $M$ and $k$ are positive constants. As usual,
we are identifying $T\mathbb{R}^{n}$ and $\mathbb{R}^{n}\times\mathbb{R}^{n}$.
For $Q=\mathbb{R}^{n}-\left\{ 0\right\} $, the action 
\begin{equation}
\widetilde{\psi}:\left(s,\mathbf{q}\right)\in\mathbb{R}^{+}\times Q\mapsto\sqrt{s}\,\mathbf{q}\in Q\label{fsx}
\end{equation}
is principal and with related bundle 
\begin{equation}
\pi:\mathbf{q}\in Q\mapsto\frac{\mathbf{q}}{\left\Vert \mathbf{q}\right\Vert }\in\left.S^{n-1}\simeq Q\right/\mathbb{R}^{+},\label{ps}
\end{equation}
where $S^{n-1}$ is the unit sphere in ${\mathbb{R}}^{n}.$ It is
clear that the restriction $\tilde{L}:TQ\rightarrow\mathbb{R}$ of
$L$ is homogeneous. An example of scaling function $f:Q\rightarrow\mathbb{R}^{+}$
is given by 
\begin{equation}
f\left(\mathbf{q}\right)=\left\Vert \mathbf{q}\right\Vert ^{2}.\label{fx}
\end{equation}

\subsection{A candidate for the \textit{reduced action}}

\label{are}

Let us go back to a general situation. Suppose that we have a homogeneous
Lagrangian function $L:TQ\rightarrow\mathbb{R}$ w.r.t. (the tangent
lift of) a principal action $\psi:\mathbb{R}^{+}\times Q\rightarrow Q$,
and fix a scaling function $f:Q\rightarrow\mathbb{R}^{+}$. In a similar
way as we have defined the reduced Lagrangian in the case of standard
symmetries, now, using $f$, the related diffeomorphism $\alpha_{f}$
(see Eq. \eqref{alfaf}) and the fact that the quotient $L/\left(f\circ\tau_{Q}\right)$
is invariant (w.r.t. the tangent lift of $\psi$), we can define the
function 
\[
\ell:T\left(\left.Q\right/\mathbb{R}^{+}\right)\times\mathbb{R}\rightarrow\mathbb{R}
\]
by the formula 
\begin{equation}
\ell\circ\alpha_{f}\circ p=\frac{L}{f\circ\tau_{Q}},\label{rel}
\end{equation}
and we shall also call it \textbf{reduced Lagrangian}. Note that,
using \eqref{ito2}, given a curve $\gamma:\left[0,\tau\right]\rightarrow Q$
we have the identity 
\[
\ell\left(\dot{x}\left(t\right),y\left(t\right)\right)=\frac{L\left(\dot{\gamma}\left(t\right)\right)}{f\left(\gamma\left(t\right)\right)},
\]
where $x:[0,\tau]\to Q/{\Bbb R}^{+}$ and $y:[0,\tau]\to{\Bbb R}$
are given by (\ref{ito1}).

In these terms, the action of $L$ for some $\tau>0$ (see Eq. \eqref{aL})
can be written as 
\[
A\left(\gamma\right)=\int_{0}^{\tau}L\left(\dot{\gamma}\left(t\right)\right)\,\mathsf{d}t=\int_{0}^{\tau}f\left(\gamma\left(t\right)\right)\,\frac{L\left(\dot{\gamma}\left(t\right)\right)}{f\left(\gamma\left(t\right)\right)}\,\mathsf{d}t=\int_{0}^{\tau}f\left(\gamma\left(t\right)\right)\,\ell\left(\dot{x}\left(t\right),y\left(t\right)\right)\,\mathsf{d}t.
\]
On the other hand, using the second part of \eqref{ito1}, we can
write 
\[
\begin{array}{lll}
f\left(\gamma\left(t\right)\right) & = & \exp\left(\ln\left(f\left(\gamma\left(t\right)\right)\right)\right)=\exp\left({\displaystyle \int_{0}^{t}\frac{d\ln\left(f\left(\gamma\left(s\right)\right)\right)}{\mathsf{d}s}}\,\mathsf{d}s+\ln\left(f\left(\gamma\left(0\right)\right)\right)\right)\\
\\
 & = & f\left(\gamma\left(0\right)\right)\,\exp\left({\displaystyle \int_{0}^{t}y\left(s\right)\,\mathsf{d}s}\right).
\end{array}
\]
As a consequence, we have that 
\begin{equation}
A\left(\gamma\right)=f\left(\gamma\left(0\right)\right)\,\check{A}\left(x,y\right),\label{paa}
\end{equation}
with 
\begin{equation}
\check{A}\left(x,y\right)\coloneqq\int_{0}^{\tau}\exp\left(\int_{0}^{t}y\left(s\right)\,\mathsf{d}s\right)\,\ell\left(\dot{x}\left(t\right),y\left(t\right)\right)\,\mathsf{d}t,\label{aLr}
\end{equation}
a functional on curves $\left\{ \left(x,y\right):\left[0,\tau\right]\rightarrow\left(\left.Q\right/\mathbb{R}^{+}\right)\times\mathbb{R}\right\} $.
Eq. \eqref{paa} says that the functionals $A$ and $\check{A}$ are
proportional. This suggests that $\check{A}$ could be a good candidate
for the reduced action we are looking for. To see that this is in
fact the case, we need the results of the next section. Among other
things, such results tell us what kind of infinitesimal variations
we must consider for the wanted reduced principle.

\subsection{The interrelated curves and variations}

For every curve $\gamma:\left[0,\tau\right]\rightarrow Q$, we have
through the diffeomorphism $\alpha_{f}$ (see Eq. \eqref{alfaf})
the curves 
\begin{equation}
x\left(t\right)=\pi\left(\gamma\left(t\right)\right)\quad\textrm{and}\quad y\left(t\right)=\mathsf{d}\ln f\left(\dot{\gamma}\left(t\right)\right)\label{cop}
\end{equation}
(see Eqs. \eqref{ito1} and \eqref{ito2}). Below, we show that any
pair of curves $x:\left[0,\tau\right]\rightarrow Q/\mathbb{R}^{+}$
and $y:\left[0,\tau\right]\rightarrow\mathbb{R}$ can be defined from
a unique curve $\gamma$ as in above equation, up to a scaling transformation. 
\begin{prop}
\label{p7} Consider three curves $\gamma:\left[0,\tau\right]\rightarrow Q$,
$x:\left[0,\tau\right]\rightarrow Q/\mathbb{R}^{+}$ and $y:\left[0,\tau\right]\rightarrow\mathbb{R}$,
and a number $\varsigma\in\mathbb{R}^{+}$. The following statements
are equivalent. 
\begin{enumerate}
\item $x$, $y$ and $\varsigma$ are given by 
\begin{equation}
x\left(t\right)=\pi\left(\gamma\left(t\right)\right),\quad y\left(t\right)=\mathsf{d}\ln f\left(\dot{\gamma}\left(t\right)\right)\quad\textrm{and}\quad\varsigma=f\left(\gamma\left(0\right)\right).\label{copp}
\end{equation}
\item $\gamma$ is given by 
\begin{equation}
\gamma\left(t\right)=\psi\left(e^{\int_{0}^{t}y\left(s\right)\,\mathsf{d}s},\left(\pi,f\right)^{-1}\left(x\left(t\right),\varsigma\right)\right).\label{cc}
\end{equation}
\end{enumerate}
\end{prop}

\begin{proof}
We shall show that, given $x$, $y$ and $\varsigma$, the unique
solution $\gamma$ for the Eq. \eqref{copp} is the curve \eqref{cc}.
On the one hand, if $\gamma$ is given by \eqref{cc}, then 
\begin{equation}
\pi\left(\gamma\left(t\right)\right)=\pi\left(\left(\pi,f\right)^{-1}\left(x\left(t\right),\varsigma\right)\right)=x\left(t\right).\label{cop1}
\end{equation}
On the other hand, using the identity (see \eqref{invt}) 
\[
\left(\pi,f\right)^{-1}\left(\pi\left(q\right),\varsigma\right)=\psi\left(\frac{\varsigma}{f\left(q\right)},q\right),
\]
from Eq. \eqref{cop1} we have that 
\[
\gamma\left(t\right)=\psi\left(e^{\int_{0}^{t}y\left(s\right)\,\mathsf{d}s},\psi\left(\frac{\varsigma}{f\left(\gamma\left(t\right)\right)},\gamma\left(t\right)\right)\right)=\psi\left(\frac{\varsigma\,e^{\int_{0}^{t}y\left(s\right)\,\mathsf{d}s}}{f\left(\gamma\left(t\right)\right)},\gamma\left(t\right)\right)=\left(\pi,f\right)^{-1}\left(x\left(t\right),\varsigma\,e^{\int_{0}^{t}y\left(s\right)\,\mathsf{d}s}\right),
\]
and consequently 
\[
f\left(\gamma\left(t\right)\right)=f\left(\left(\pi,f\right)^{-1}\left(x\left(t\right),\varsigma\,e^{\int_{0}^{t}y\left(s\right)\,\mathsf{d}s}\right)\right)=\varsigma\,e^{\int_{0}^{t}y\left(s\right)\,\mathsf{d}s}.
\]
Note in particular that $f\left(\gamma\left(0\right)\right)=\varsigma$.
Differentiating with respect to $t$, we have that 
\[
df\left(\dot{\gamma}\left(t\right)\right)=\varsigma\,e^{\int_{0}^{t}y\left(s\right)\,\mathsf{d}s}\;y\left(t\right)=f\left(\gamma\left(t\right)\right)\;y\left(t\right),
\]
so 
\[
y\left(t\right)=\frac{df\left(\dot{\gamma}\left(t\right)\right)}{f\left(\gamma\left(t\right)\right)}=\mathsf{d}\ln f\left(\dot{\gamma}\left(t\right)\right).
\]
Then, the curve \eqref{cc} is a solution of \eqref{copp}.

Suppose we have another solution $\lambda:\left[0,\tau\right]\rightarrow Q$.
Then, in particular, 
\begin{equation}
\pi\left(\lambda\left(t\right)\right)=\pi\left(\gamma\left(t\right)\right)\quad\textrm{and}\quad\mathsf{d}\ln f\left(\dot{\lambda}\left(t\right)\right)=\mathsf{d}\ln f\left(\dot{\gamma}\left(t\right)\right).\label{pdp}
\end{equation}
The first identity in \eqref{pdp} says that $\lambda\left(t\right)=\psi\left(g\left(t\right),\gamma\left(t\right)\right)$
for some function $g:\left[0,\tau\right]\rightarrow\mathbb{R}^{+}$.
Thus, 
\begin{equation}
f\left(\lambda\left(t\right)\right)=g\left(t\right)\,f\left(\gamma\left(t\right)\right),\label{rdfr}
\end{equation}
and consequently 
\[
\begin{array}{lll}
\mathsf{d}\ln f\left(\dot{\lambda}\left(t\right)\right) & = & {\displaystyle \frac{\mathsf{d}(f(\dot{\lambda}(t))}{f(\lambda(t))}=\frac{{\displaystyle \frac{d}{dt}(f(\lambda(t)))}}{f(\lambda(t))}=\frac{\dot{g}\left(t\right)\,f\left(\gamma\left(t\right)\right)+g\left(t\right){\displaystyle \frac{d}{dt}(f(\gamma(t)))}}{g(t)\,f(\gamma(t))}}\\
\\
 & = & {\displaystyle \frac{\dot{g}\left(t\right)}{g\left(t\right)}+\mathsf{d}\ln f\left(\dot{\gamma}\left(t\right)\right).}
\end{array}
\]
This implies, according to the second identity of \eqref{pdp}, that
$\dot{g}\left(t\right)=0$. On the other hand, because of \eqref{copp},
\[
f\left(\lambda\left(0\right)\right)=\varsigma=f\left(\gamma\left(0\right)\right).
\]
So, it follows from \eqref{rdfr} that $g\left(t\right)=1$ for all
$t$, and consequently $\lambda=\gamma$. 
\end{proof}
\begin{defn}
\label{defi} Consider three curves $\gamma$, $x$ and $y$, and
a positive number $\varsigma$, as in the previous proposition. We
shall say these four objects are \textbf{interrelated} if they satisfy
some of the equivalent conditions of the same proposition. 
\end{defn}

Now, let us study the relationship between the variations of a curve
$\gamma$ and those of its interrelated curves $x,y$. 
\begin{prop}
\label{p8} Consider three curves $\gamma:\left[0,\tau\right]\rightarrow Q$,
$x:\left[0,\tau\right]\rightarrow Q/\mathbb{R}^{+}$ and $y:\left[0,\tau\right]\rightarrow\mathbb{R}$,
and a number $\varsigma\in\mathbb{R}^{+}$, which are interrelated.
Consider also functions 
\[
\Gamma:\left(-\epsilon,\epsilon\right)\times\left[0,\tau\right]\rightarrow Q,\quad\Gamma_{x}:\left(-\epsilon,\epsilon\right)\times\left[0,\tau\right]\rightarrow Q/\mathbb{R}^{+},\quad\Gamma_{y}:\left(-\epsilon,\epsilon\right)\times\left[0,\tau\right]\rightarrow\mathbb{R}
\]
and 
\[
\chi:\left(-\epsilon,\epsilon\right)\rightarrow\mathbb{R}^{+}
\]
such that, for each $s\in\left(-\epsilon,\epsilon\right)$, the curves
$\Gamma\left(s,\cdot\right)$, $\Gamma_{x}\left(s,\cdot\right)$ and
$\Gamma_{y}\left(s,\cdot\right)$, and the number $\chi\left(s\right)\in\mathbb{R}^{+}$,
are interrelated, i.e. (for each $s$) 
\begin{equation}
\Gamma_{x}\left(s,t\right)=\pi\left(\Gamma\left(s,t\right)\right),\quad\Gamma_{y}\left(s,t\right)=\mathsf{d}\ln f\left(\frac{\partial}{\partial t}\Gamma\left(s,t\right)\right)\quad\textrm{and}\quad\chi\left(s\right)=f\left(\Gamma\left(s,0\right)\right),\label{gpg}
\end{equation}
or equivalently 
\begin{equation}
\Gamma\left(s,t\right)=\psi\left(e^{\int_{0}^{t}\Gamma_{y}\left(s,r\right)\,dr},\left(\pi,f\right)^{-1}\left(\Gamma_{x}\left(s,t\right),\chi\left(s\right)\right)\right).\label{gp9}
\end{equation}
Then, 
\begin{equation}
\delta\gamma\left(t\right)=\left.\frac{\partial}{\partial s}\right|_{s=0}\Gamma\left(s,t\right)\label{cg}
\end{equation}
is a variation of $\gamma$ satisfying \eqref{eLv} if and only if
\begin{equation}
\delta x\left(t\right)=\left.\frac{\partial}{\partial s}\right|_{s=0}\Gamma_{x}\left(s,t\right)\quad\textrm{and }\quad\delta y\left(t\right)=\left.\frac{\partial}{\partial s}\right|_{s=0}\Gamma_{y}\left(s,t\right)\label{ggy}
\end{equation}
are variations of $x$ and $y$, respectively, satisfying 
\begin{equation}
\delta x\left(0\right),\delta x\left(\tau\right)\in\mathsf{O}_{Q/\mathbb{R}^{+}}\quad\textrm{and}\quad\int_{0}^{\tau}\delta y\left(t\right)\,\mathsf{d}t=0,\label{eLvr}
\end{equation}
and also 
\begin{equation}
\chi\left(0\right)=\varsigma,\quad\chi'\left(0\right)=0.\label{fsf}
\end{equation}
\end{prop}

\begin{proof}
Suppose that $\Gamma$ defines a variation $\delta\gamma$ of $\gamma$
satisfying \eqref{eLv}. Consider the functions $\Gamma_{x}$ and
$\Gamma_{y}$ given by \eqref{gpg}. It is easy to see from \eqref{cg}
and \eqref{ggy} that $\delta x\left(t\right)=\pi_{*}\left(\delta\gamma\left(t\right)\right)$,
and from \eqref{copp} that $\delta x$ is a variation of $x$. The
fact that $\delta x$ fulfills \eqref{eLvr} is immediate from \eqref{eLv}.
On the other hand (see \eqref{gpg} and \eqref{ggy} again), writing
$\lambda=\ln f$ for simplicity, 
\[
\begin{array}{lll}
{\displaystyle \int_{0}^{\tau}\delta y\left(t\right)\,\mathsf{d}t} & = & {\displaystyle \int_{0}^{\tau}\left.\frac{\partial}{\partial s}\right|_{s=0}\mathsf{d}\lambda\left(\frac{\partial}{\partial t}\Gamma\left(s,t\right)\right)\,\mathsf{d}t=\left.\frac{\partial}{\partial s}\right|_{s=0}\int_{0}^{\tau}\frac{\partial}{\partial t}\lambda\left(\Gamma\left(s,t\right)\right)\,\mathsf{d}t}\\
\\
 & = & {\displaystyle \left.\frac{\partial}{\partial s}\right|_{s=0}\left(\lambda\left(\Gamma\left(s,\tau\right)\right)-\lambda\left(\Gamma\left(s,0\right)\right)\right)=\mathsf{d}\lambda\left(\delta\gamma\left(\tau\right)\right)-\mathsf{d}\lambda\left(\delta\gamma\left(0\right)\right)=0},
\end{array}
\]
where the last equality is a consequence of \eqref{eLv}. Also, since
\[
\Gamma_{y}\left(0,t\right)=\mathsf{d}\lambda\left(\frac{\partial}{\partial t}\Gamma\left(0,t\right)\right)=\mathsf{d}\lambda\left(\dot{\gamma}\left(t\right)\right)=y\left(t\right)
\]
(see \eqref{copp}), we have that $\delta y$ is a variation of $y$
fulfilling \eqref{eLvr}. Finally, since 
\[
\chi(0)=f\left(\Gamma\left(0,0\right)\right)=f\left(\gamma\left(0\right)\right)=\varsigma
\]
and 
\[
\chi'(0)=\left.\frac{\partial}{\partial s}\right|_{s=0}f\left(\Gamma\left(s,0\right)\right)=\mathsf{d}f\left(\delta\gamma\left(0\right)\right)=0,
\]
Eq. \eqref{fsf} follows.

Now, let us show the converse. Suppose that $\Gamma_{x}$ and $\Gamma_{y}$
define variations $\delta x$ and $\delta y$ satisfying \eqref{eLvr},
and $\chi$ is a function satisfying \eqref{fsf}. We want to show
that the function $\Gamma$ given by \eqref{gp9} defines a variation
$\delta\gamma$ of $\gamma$ satisfying \eqref{eLv}. It is clear
that $\Gamma\left(0,t\right)=\gamma\left(t\right)$ for all $t\in\left[0,\tau\right]$
(recall \eqref{cc} and the fact that $\chi\left(0\right)=\varsigma$).
On the other hand, since 
\[
\left.\frac{\partial}{\partial s}\right|_{s=0}\Gamma\left(s,t\right)=\left(\psi_{e^{\int_{0}^{t}y\left(r\right)\,dr}}\right)_{*}\left(\left(\pi,f\right)_{*}^{-1}\left(\delta x\left(t\right),\chi'\left(0\right)\right)\right)+\left(\psi_{\left(\pi,f\right)^{-1}\left(x\left(t\right),\varsigma\right)}\right)_{*}\left(e^{\int_{0}^{t}y\left(r\right)\,dr}\,\int_{0}^{t}\delta y\left(r\right)\,dr\right),
\]
if we replace in the last expression $t$ by $0$ or $\tau$, and
use that $\chi'\left(0\right)=0$, we have from \eqref{eLvr} that
Eq. \eqref{eLv} holds, which ends our proof. 
\end{proof}
The previous result suggests that the variations which we need for
the desired reduced variational principle must satisfy \eqref{eLvr}. 
\begin{rem*}
Note that the condition for the variation $\delta y$ appearing in
Eq. \eqref{eLvr} is the same as that appearing for standard (abelian)
symmetries (see Eq. \eqref{oc}). $\quad\diamond$ 
\end{rem*}

\subsection{The reduction and reconstruction processes}

Now, let us study the relationship between the critical points of
$A$ and $\check{A}$ (see Eqs. \eqref{aL} and \eqref{aLr}) w.r.t.
appropriate variations. 
\begin{prop}
\label{p11} Suppose we have curves $\gamma$, $x$ and $y$, and
a positive number $\varsigma$, which are interrelated. Then, 
\begin{equation}
A\left(\gamma\right)=\varsigma\,\check{A}\left(x,y\right).\label{ep11}
\end{equation}
\end{prop}

\begin{proof}
We have shown at Section \ref{are} that $A\left(\gamma\right)=f\left(\gamma\left(0\right)\right)\,\check{A}\left(x,y\right)$,
if $x\left(t\right)=\pi\left(\gamma\left(t\right)\right)$ and $y\left(t\right)=\mathsf{d}\ln f\left(\dot{\gamma}\left(t\right)\right)$.
If in addition $f\left(\gamma\left(0\right)\right)=\varsigma$ (see
\eqref{copp}), we have the result that we wanted. 
\end{proof}
\begin{rem}
\label{rsg} Note that, if the curves $\Gamma\left(s,\cdot\right)$,
$\Gamma_{x}\left(s,\cdot\right)$ and $\Gamma_{y}\left(s,\cdot\right)$,
and a number $\chi\left(s\right)\in\mathbb{R}^{+}$, are interrelated
for each $s$, then, according to the last Proposition, 
\begin{equation}
A\left(\Gamma\left(s,\cdot\right)\right)=\chi\left(s\right)\,\check{A}\left(\Gamma_{x}\left(s,\cdot\right),\Gamma_{y}\left(s,\cdot\right)\right).\quad\diamond\label{sgn}
\end{equation}
\end{rem}

Now, we can enunciate and prove the main result of the paper. 
\begin{thm}
\label{t3} Consider curves $\gamma$, $x$ and $y$, and a number
$\varsigma\in\mathbb{R}^{+}$, which are interrelated. Then, the curve
$\gamma$ is a critical point of the functional 
\[
A\left(\gamma\right)=\int_{0}^{\tau}L\left(\dot{\gamma}\left(t\right)\right)\,\mathsf{d}t
\]
with respect to variations satisfying 
\[
\delta\gamma\left(0\right),\delta\gamma\left(\tau\right)\in\mathsf{O}_{Q},
\]
if and only if the curves $x$ and $y$ define a critical point of
the functional 
\begin{equation}
\check{A}\left(x,y\right)=\int_{0}^{\tau}\exp\left(\int_{0}^{t}y\left(s\right)\,\mathsf{d}s\right)\,\ell\left(\dot{x}\left(t\right),y\left(t\right)\right)\,\mathsf{d}t\label{rvp1}
\end{equation}
with respect to variations satisfying 
\begin{equation}
\delta x\left(0\right),\delta x\left(\tau\right)\in\mathsf{O}_{Q/\mathbb{R}^{+}}\quad\textrm{and}\quad\int_{0}^{\tau}\delta y\left(s\right)\,\mathsf{d}s=0.\label{rv1}
\end{equation}
\end{thm}

\begin{proof}
Suppose that $\gamma:\left[0,\tau\right]\rightarrow Q$ is a critical
point of $A$, define $\varsigma\coloneqq f\left(\gamma\left(0\right)\right)$
and consider the curves $x$ and $y$ given by \eqref{cop}. In particular,
we are saying that $\gamma$, $x$, $y$ and $\varsigma$ are interrelated.
Now, consider arbitrary variations $\delta x$ and $\delta y$ of
$x$ and $y$, satisfying \eqref{eLvr}, given by functions $\Gamma_{x}$
and $\Gamma_{y}$, and consider also a function $\chi$ satisfying
\eqref{fsf}. Define with them a function $\Gamma$ as in Eq. \eqref{gp9}.
On the one hand, according to Remark \ref{rsg}, we have that Eq.
\eqref{sgn} holds. And, on the other hand, according to Proposition
\ref{p8}, $\Gamma$ defines a variation $\delta\gamma$ of $\gamma$
that satisfies \eqref{eLv}. Then, 
\[
0=\delta A\left(\gamma\right)=\left.\frac{\partial}{\partial s}\right|_{s=0}A\left(\Gamma\left(s,\cdot\right)\right)=\left.\frac{\partial}{\partial s}\right|_{s=0}\chi\left(s\right)\,\check{A}\left(\Gamma_{x}\left(s,\cdot\right),\Gamma_{y}\left(s,\cdot\right)\right)=\varsigma\,\delta\check{A}\left(x,y\right),
\]
implying that $x$ and $y$ are critical for $\check{A}$ w.r.t. variations
satisfying \eqref{eLvr}. The reciprocal can be proved in a similar
way. 
\end{proof}
Summing up, above result defines a \textbf{reduced variational principle}
for the homogeneous Lagrangian system $\left(Q,L\right)$ with action
$\psi$ and scaling function $f$: the \textbf{reduced action} is
given by \eqref{rvp1} and the \textbf{reduced variations} must satisfy
\eqref{rv1}. Also, the result gives us a \textbf{reconstruction process}.
In fact, given a critical point $\left(x,y\right)$ of the reduced
principle and any positive number $\varsigma$, the interrelated curve
(compare with \eqref{gtif}) 
\begin{equation}
\gamma\left(t\right)=\psi\left(e^{\int_{0}^{t}y\left(s\right)\,\mathsf{d}s},\left(\pi,f\right)^{-1}\left(x\left(t\right),\varsigma\right)\right)\label{recp}
\end{equation}
is a critical point of the Hamilton principle for $L$, i.e. a trajectory
of our homogeneous Lagrangian system $\left(Q,L\right)$. 
\begin{rem*}
Moreover, given $q_{0}\in Q$ and $v_{0}\in T_{q_{0}}Q$, if there
exists a critical point $\left(x,y\right)$ such that $x\left(0\right)=\pi\left(q_{0}\right)$,
$\dot{x}\left(0\right)=\pi_{*}\left(v_{0}\right)$ and $y\left(0\right)=d(\ln f)(v_{0})=\left.\mathsf{d}f\left(v_{0}\right)\right/{\displaystyle f(q_{0})}$,
then, choosing $\varsigma=f\left(q_{0}\right)$, it can be shown that
$\dot{\gamma}\left(0\right)=v_{0}$ (\textit{ipso facto} $\gamma\left(0\right)=q_{0}$).
$\quad\diamond$ 
\end{rem*}
All of this is, precisely, what we have been looking for.

\subsection{The \textit{scaling}-Lagrange-Poincaré equations}

In this section we shall study the critical points of the variational
principle given by \eqref{rvp1} and \eqref{rv1}. We shall see they
satisfy a set of ordinary differential equations with a similar structure
to the Lagrange-Poincaré equations (see Section \ref{LPe}). To do
that, consider an affine connection $\nabla$ on $T\left(Q/\mathbb{R}^{+}\right)$,
the related diffeomorphism 
\[
\Omega:TT\left(Q/\mathbb{R}^{+}\right)\rightarrow T\left(Q/\mathbb{R}^{+}\right)\oplus T\left(Q/\mathbb{R}^{+}\right)\oplus T\left(Q/\mathbb{R}^{+}\right)
\]
and the linear isomorphisms $\Omega_{\dot{x}}$, for all $\dot{x}\in T\left(Q/\mathbb{R}^{+}\right)$
(see Section \ref{affcon}). For the reduced Lagrangian $\ell:T\left(Q/\mathbb{R}^{+}\right)\times\mathbb{R}\rightarrow\mathbb{R}$,
repeat the definitions made along the Eqs. \eqref{dfb1} to \eqref{dfb4},
but for $G=\mathbb{R}^{+}$. With this notation, we have the next
theorem. 
\begin{thm}
\label{t35} The curves $x$ and $y$ define a critical point of the
functional \eqref{rvp1} with respect to variations satisfying \eqref{rv1}
if and only if they satisfy the \textbf{horizontal scaling-Lagrange-Poincaré
equation} 
\begin{equation}
-\frac{D}{Dt}\frac{\partial\ell}{\partial\dot{x}}\left(\dot{x}\left(t\right),y\left(t\right)\right)-y\left(t\right)\,\frac{\partial\ell}{\partial\dot{x}}\left(\dot{x}\left(t\right),y\left(t\right)\right)+\frac{\partial\ell}{\partial x}\left(\dot{x}\left(t\right),y\left(t\right)\right)=0\label{hor}
\end{equation}
and the \textbf{vertical scaling-Lagrange-Poincaré equation} 
\begin{equation}
-\frac{d}{dt}\frac{\partial\ell}{\partial y}\left(\dot{x}\left(t\right),y\left(t\right)\right)-y\left(t\right)\,\frac{\partial\ell}{\partial y}\left(\dot{x}\left(t\right),y\left(t\right)\right)+\ell\left(\dot{x}\left(t\right),y\left(t\right)\right)=0,\label{ver}
\end{equation}
for all $t\in\left(0,\tau\right)$. 
\end{thm}

\begin{proof}
It is easy to see that (recall the Remark \ref{invar}) 
\begin{equation}
\begin{array}{lll}
\delta\check{A}\left(x,y\right) & = & \int_{0}^{\tau}\exp\left(\int_{0}^{t}y\left(s\right)\,\mathsf{d}s\right)\,\left(\int_{0}^{t}\delta y\left(s\right)\,\mathsf{d}s\right)\,\ell\left(\dot{x}\left(t\right),y\left(t\right)\right)\,dt\\
\\
 &  & +\int_{0}^{\tau}\exp\left(\int_{0}^{t}y\left(s\right)\,\mathsf{d}s\right)\,\delta\left(\ell\left(\dot{x}\left(t\right),y\left(t\right)\right)\right)\,dt
\end{array}\label{dS-1}
\end{equation}
and 
\[
\delta\left(\ell\left(\dot{x}\left(t\right),y\left(t\right)\right)\right)=\left\langle d_{1}\ell\left(\dot{x}\left(t\right),y\left(t\right)\right),\delta\dot{x}\left(t\right)\right\rangle +\frac{\partial\ell}{\partial y}\left(\dot{x}\left(t\right),y\left(t\right)\right)\,\delta y\left(t\right),
\]
and, from the definition of ${\displaystyle \left.\partial\ell\right/\partial x}$
and ${\displaystyle \left.\partial\ell\right/\partial\dot{x}}$ (recall
also the Eq. \eqref{Ov}), 
\[
\left\langle d_{1}\ell\left(\dot{x}\left(t\right),y\left(t\right)\right),\delta\dot{x}\left(t\right)\right\rangle ={\displaystyle \left\langle \frac{\partial\ell}{\partial x}\left(\dot{x}\left(t\right),y\left(t\right)\right),\delta x\left(t\right)\right\rangle +\left\langle \frac{\partial\ell}{\partial\dot{x}}\left(\dot{x}\left(t\right),y\left(t\right)\right),\frac{D}{Dt}\delta x\left(t\right)\right\rangle }.
\]
To simplify the notation, let us omit the argument $\left(\dot{x}\left(t\right),y\left(t\right)\right)$
in $\ell$, ${\displaystyle \left.\partial\ell\right/\partial y}$,
${\displaystyle \left.\partial\ell\right/\partial x}$ and ${\displaystyle \left.\partial\ell\right/\partial\dot{x}}$
. On the one hand, since 
\[
\frac{d}{dt}\left\langle \frac{\partial\ell}{\partial\dot{x}},\delta x\left(t\right)\right\rangle -\left\langle \frac{D}{Dt}\frac{\partial\ell}{\partial\dot{x}},\delta x\left(t\right)\right\rangle =\left\langle \frac{\partial\ell}{\partial\dot{x}},\frac{D}{Dt}\delta x\left(t\right)\right\rangle ,
\]
and 
\[
\begin{array}{lll}
{\displaystyle \frac{d}{dt}\left(\exp\left(\int_{0}^{t}y\left(s\right)\,\mathsf{d}s\right)\,\left\langle \frac{\partial\ell}{\partial\dot{x}},\delta x\left(t\right)\right\rangle \right)} & - & {\displaystyle \exp\left(\int_{0}^{t}y\left(s\right)\,\mathsf{d}s\right)\,\left(y\left(t\right)\,\left\langle \frac{\partial\ell}{\partial\dot{x}},\delta x\left(t\right)\right\rangle \right)}\\
\\
 & = & {\displaystyle \exp\left(\int_{0}^{t}y\left(s\right)\,\mathsf{d}s\right)\,\left(\frac{d}{dt}\left\langle \frac{\partial\ell}{\partial\dot{x}},\delta x\left(t\right)\right\rangle \right),}
\end{array}
\]
the terms in the integrand of $\delta\check{A}\left(x,y\right)$ proportional
to $\delta x$ are 
\[
\begin{array}{l}
{\displaystyle \exp\left(\int_{0}^{t}y\left(s\right)\,\mathsf{d}s\right)\,\left\langle \frac{\partial\ell}{\partial x},\delta x\left(t\right)\right\rangle }+{\displaystyle \frac{d}{dt}\left(\exp\left(\int_{0}^{t}y\left(s\right)\,\mathsf{d}s\right)\,\left\langle \frac{\partial\ell}{\partial\dot{x}},\delta x\left(t\right)\right\rangle \right)}\\
\\
{\displaystyle -\exp\left(\int_{0}^{t}y\left(s\right)\,\mathsf{d}s\right)\,\left(y\left(t\right)\,\left\langle \frac{\partial\ell}{\partial\dot{x}},\delta x\left(t\right)\right\rangle \right)}-{\displaystyle \exp\left(\int_{0}^{t}y\left(s\right)\,\mathsf{d}s\right)\,\left\langle \frac{D}{Dt}\frac{\partial\ell}{\partial\dot{x}},\delta x\left(t\right)\right\rangle .}
\end{array}
\]
On the other hand, since 
\[
\begin{array}{lll}
{\displaystyle \frac{d}{dt}\left(\exp\left(\int_{0}^{t}y\left(s\right)\,\mathsf{d}s\right)\,\frac{\partial\ell}{\partial y}\,\int_{0}^{t}\delta y\left(s\right)\,\mathsf{d}s\right)} & - & {\displaystyle \exp\left(\int_{0}^{t}y\left(s\right)\,\mathsf{d}s\right)\,{\displaystyle \left(y\left(t\right)\,\frac{\partial\ell}{\partial y}\,\int_{0}^{t}\delta y\left(s\right)\,\mathsf{d}s\right)}}\\
\\
 & - & {\displaystyle \exp\left(\int_{0}^{t}y\left(s\right)\,\mathsf{d}s\right)\,\left(\frac{d}{dt}\left(\frac{\partial\ell}{\partial y}\right)\,\int_{0}^{t}\delta y\left(s\right)\,\mathsf{d}s\right)}\\
\\
 & = & {\displaystyle \exp\left(\int_{0}^{t}y\left(s\right)\,\mathsf{d}s\right)\,\left(\frac{\partial\ell}{\partial y}\,\delta y\left(t\right)\right)},
\end{array}
\]
the terms in the integrand of $\delta\check{A}\left(x,y\right)$ proportional
to $\delta y$ are 
\[
\begin{array}{l}
{\displaystyle \frac{d}{dt}\left(\exp\left(\int_{0}^{t}y\left(s\right)\,\mathsf{d}s\right)\,\frac{\partial\ell}{\partial y}\,\int_{0}^{t}\delta y\left(s\right)\,\mathsf{d}s\right)}-{\displaystyle \exp\left(\int_{0}^{t}y\left(s\right)\,\mathsf{d}s\right)\,{\displaystyle \left(y\left(t\right)\,\frac{\partial\ell}{\partial y}\,\int_{0}^{t}\delta y\left(s\right)\,\mathsf{d}s\right)}}\\
\\
-{\displaystyle \exp\left(\int_{0}^{t}y\left(s\right)\,\mathsf{d}s\right)\,\left(\frac{d}{dt}\left(\frac{\partial\ell}{\partial y}\right)\,\int_{0}^{t}\delta y\left(s\right)\,\mathsf{d}s\right)}+{\displaystyle \exp\left(\int_{0}^{t}y\left(s\right)\,\mathsf{d}s\right)\,\left(\int_{0}^{t}\delta y\left(s\right)\,\mathsf{d}s\right)\,\ell}.
\end{array}
\]
Finally, using that the terms with the total derivatives $d/dt$ vanish
(because of \eqref{rv1}), and the fact that the variations $\delta x$
and $\delta y$ are independent, we have that the critical points
must satisfy 
\begin{equation}
\int_{0}^{\tau}\exp\left(\int_{0}^{t}y\left(s\right)\,\mathsf{d}s\right)\left\langle \frac{D}{Dt}\frac{\partial\ell}{\partial\dot{x}}\left(\dot{x}\left(t\right),y\left(t\right)\right)+y\left(t\right)\,\frac{\partial\ell}{\partial\dot{x}}\left(\dot{x}\left(t\right),y\left(t\right)\right)-\frac{\partial\ell}{\partial x}\left(\dot{x}\left(t\right),y\left(t\right)\right),\delta x\left(t\right)\right\rangle \,dt=0\label{XdX}
\end{equation}
and 
\begin{equation}
\int_{0}^{\tau}\exp\left(\int_{0}^{t}y\left(s\right)\,\mathsf{d}s\right)\left(\frac{d}{dt}\frac{\partial\ell}{\partial y}\left(\dot{x}\left(t\right),y\left(t\right)\right)+y\left(t\right)\,\frac{\partial\ell}{\partial y}\left(\dot{x}\left(t\right),y\left(t\right)\right)-\ell\left(\dot{x}\left(t\right),y\left(t\right)\right)\right)\,\left(\int_{0}^{t}\delta y\left(s\right)\,\mathsf{d}s\right)\,dt=0.\label{YdY}
\end{equation}
So, using standard techniques of variational calculus, Eq. \eqref{hor}
easily follows from \eqref{XdX}. To prove \eqref{ver}, we need to
work a little bit more. Suppose that 
\[
G\left(t\right)\coloneqq\frac{d}{dt}\frac{\partial\ell}{\partial y}\left(\dot{x}\left(t\right),y\left(t\right)\right)+y\left(t\right)\,\frac{\partial\ell}{\partial y}\left(\dot{x}\left(t\right),y\left(t\right)\right)-\ell\left(\dot{x}\left(t\right),y\left(t\right)\right)>0
\]
for some $t_{0}\in\left(0,\tau\right)$. By continuity, there must
exist $\epsilon>0$ such that $\left(t_{0}-\epsilon,t_{0}+\epsilon\right)\subseteq\left(0,\tau\right)$
and above inequality holds for all $t\in\left(t_{0}-\epsilon,t_{0}+\epsilon\right)$.
Let $\varphi:\mathbb{R}\rightarrow\mathbb{R}$ be a positive smooth
function, even w.r.t. $t_{0}$ and with support equal to $\left[t_{0}-\epsilon,t_{0}+\epsilon\right]$.
Define 
\begin{equation}
\delta y\left(s\right)\coloneqq\varphi\left(s\right)\,\sin\left(\frac{t_{0}-s}{\epsilon}\,\pi\right),\quad s\in\left[0,\tau\right].\label{vdy}
\end{equation}
Then, we have that, 
\[
\tilde{\varphi}\left(t\right)\coloneqq\int_{0}^{t}\delta y\left(s\right)\,\mathsf{d}s=\int_{0}^{t}\varphi\left(s\right)\,\sin\left(\frac{t_{0}-s}{\epsilon}\,\pi\right)\,\mathsf{d}s=\left\{ \begin{array}{ll}
0, & \quad0\leq t<t_{0}-\epsilon,\\
{\displaystyle \int_{t_{0}-\epsilon}^{t}\varphi\left(s\right)\,\sin\left({\displaystyle \frac{t_{0}-s}{\epsilon}\,\pi}\right)\,\mathsf{d}s,} & \quad t_{0}-\epsilon\leq t\leq t_{0}+\epsilon,\\
0, & \quad t_{0}+\epsilon<t\leq\tau.
\end{array}\right.
\]
In particular, $\tilde{\varphi}\left(t\right)>0$ for $t\in\left(t_{0}-\epsilon,t_{0}+\epsilon\right)$
and it is null outside. Then, the variation $\delta y$ given by \eqref{vdy}
satisfies \eqref{rv1} and 
\[
\begin{array}{l}
{\displaystyle \int_{0}^{\tau}\exp\left(\int_{0}^{t}y\left(s\right)\,\mathsf{d}s\right)\,\left(\int_{0}^{t}\delta y\left(s\right)\,\mathsf{d}s\right)\,G\left(t\right)\,dt={\displaystyle \int_{0}^{\tau}\exp\left(\int_{0}^{t}y\left(s\right)\,\mathsf{d}s\right)\,\tilde{\varphi}\left(t\right)\,G\left(t\right)\,dt}}\\
\\
={\displaystyle \int_{t_{0}-\epsilon}^{t_{0}+\epsilon}\underbrace{\exp\left(\int_{0}^{t}y\left(s\right)\,\mathsf{d}s\right)\,\tilde{\varphi}\left(t\right)\,G\left(t\right)}_{>0}\,dt>0,}
\end{array}
\]
which contradicts Eq. \eqref{YdY}. This proves that Eq. \eqref{ver}
must hold. 
\end{proof}
\begin{rem*}
As in Section \ref{LPe}, if we consider local coordinates $(x^{i})$
on $Q/{\Bbb R}^{+},$ the corresponding local coordinates $(x^{i},\dot{x}^{i})$
on $T(Q/{\mathbb{R}^{+}})$ and the corresponding flat local affine
connection on $T(Q/{\mathbb{R}}^{+})\to Q/{\mathbb{R}}^{+}$, then
Eqs. (\ref{hor}) and (\ref{ver}) may be expressed as 
\[
\frac{d}{dt}\left(\frac{\partial\ell}{\partial\dot{x}}\right)+y\,\frac{\partial\ell}{\partial\dot{x}}-\frac{\partial\ell}{\partial{x}}=0,\,\;\;\;\frac{d}{dt}\left(\frac{\partial\ell}{\partial y}\right)+y\,\frac{\partial\ell}{\partial y}-\ell=0,
\]
appeared in Ref. \cite{grab}, but there, remarkably, were derived
in a very different way. $\quad\diamond$ 
\end{rem*}
It is worth mentioning that, even though the (standard) Lagrange-Poincaré
equations (see Eqs. \eqref{ehor} and \eqref{ever}) and their scaling
version (see Eqs. \eqref{hor} and \eqref{ver}) have a similar structure,
they are notably different.

\subsection{An illustrative example}

\label{hlse}

Consider the example given at the end of Section \ref{jm}, where
$Q=\mathbb{R}^{+}\times\mathbb{R}^{+}$, the action $\widetilde{\psi}$
is given by \eqref{sab} and the Lagrangian is of the form $L={\left(e-V\circ\tau_{Q}\right)\,K_{g}}.$
Now, take $g$ as the Euclidean metric in $\mathbb{R}^{+}\times\mathbb{R}^{+}$,
\[
V\left(a,b\right)=\arctan\left(b/a\right),\quad\forall a,b\in\mathbb{R}^{+},\quad\textrm{and}\quad e>\pi/2.
\]
Then, 
\begin{equation}
L\left(a,b,\dot{a},\dot{b}\right)={\left(e-\arctan\left(b/a\right)\right)\,\left(\dot{a}^{2}+\dot{b}^{2}\right)}.\label{L}
\end{equation}
Since we can identify the quotient $Q/\mathbb{R}^{+}$ with the open
interval $\left(0,\pi/2\right)$, then the principal bundle can be
seen as the map $\pi\left(a,b\right)=\arctan\left(b/a\right)$. So,
if we take $f\left(a,b\right)=\left(a^{2}+b^{2}\right)/2$ as scaling
function, the Atiyah diffeomorphism is given by 
\begin{equation}
\alpha_{f}\left(p\left(a,b,\dot{a},\dot{b}\right)\right)=\left(\pi_{*,\left(a,b\right)}\left(\dot{a},\dot{b}\right),\frac{\mathsf{d}f_{\left|\left(a,b\right)\right.}\left(\dot{a},\dot{b}\right)}{f\left(a,b\right)}\right)=\left(x,\dot{x},y\right),\label{if}
\end{equation}
with 
\[
x=\arctan\left(b/a\right),\quad\dot{x}=\frac{a\,\dot{b}-b\,\dot{a}}{a^{2}+b^{2}}\quad\textrm{and}\quad y=\frac{2(a\,\dot{a}+b\,\dot{b})}{a^{2}+b^{2}}.
\]
The last two equations imply that 
\[
\dot{a}=\frac{a\,y-2b\,\dot{x}}{2}\quad\textrm{and}\quad\dot{b}=\frac{b\,y+2a\,\dot{x}}{2},
\]
and consequently 
\[
\frac{\dot{a}^{2}+\dot{b}^{2}}{a^{2}+b^{2}}=\dot{x}^{2}+\left(\frac{y}{2}\right)^{2}.
\]
Combining \eqref{rel}, \eqref{L}, \eqref{if} and the last equation,
we have that the reduced Lagrangian is 
\begin{equation}
\ell\left(x,\dot{x},y\right)=\frac{L\left(a,b,\dot{a},\dot{b}\right)}{f\left(a,b\right)}={2\left(e-x\right)\,\left(\dot{x}^{2}+\left(\frac{y}{2}\right)^{2}\right)}.\label{rll}
\end{equation}
Finally, suppose we have a solution $\left(x\left(t\right),y\left(t\right)\right)$
of the scaling-Lagrange-Poincaré equations (see \eqref{hor} and \eqref{ver})
related to $\ell$. The reconstruction formula \eqref{recp} says
that, for all $\varsigma\in\mathbb{R}^{+}$, 
\[
\left(a\left(t\right),b\left(t\right)\right)=\psi\left(e^{\int_{0}^{t}y\left(s\right)\,\mathsf{d}s},\left(\pi,f\right)^{-1}\left(x\left(t\right),\varsigma\right)\right)
\]
is a solution of the Euler-Lagrange equations for $L$. To have a
more concrete expression, let us calculate $\left(\pi,f\right)^{-1}$.
Since 
\[
\left(\pi,f\right)\left(a,b\right)=\left(\arctan\left(b/a\right),\frac{1}{2}({a^{2}+b^{2}})\right)\eqqcolon\left(c,d\right),
\]
then $\tan c=b/a$ and $\sqrt{2d}=\sqrt{a^{2}+b^{2}}$, and solving
these equations for $a$ and $b$ we have that 
\[
\left(\pi,f\right)^{-1}\left(c,d\right)=\left(\frac{\sqrt{2d}}{\sqrt{1+\left(\tan c\right)^{2}}},\frac{\sqrt{2d}}{\sqrt{1+\left(\cot c\right)^{2}}}\right).
\]
As a consequence, using \eqref{sab}, the reconstructed curve is 
\[
\left(a\left(t\right),b\left(t\right)\right)=\left(\sqrt{\frac{2\,\varsigma\,e^{\int_{0}^{t}y\left(s\right)\,\mathsf{d}s}}{1+\left(\tan\left(x\left(t\right)\right)\right)^{2}}},\sqrt{\frac{2\,\varsigma\,e^{\int_{0}^{t}y\left(s\right)\,\mathsf{d}s}}{{1+\left(\cot\left(x\left(t\right)\right)\right)^{2}}}}\right).
\]

\section{Homogeneous Lagrangians and the Herglotz principle}

\label{S5}

The \textit{Herglotz variational principle} (see \cite{lainz} and
references therein) is intimately related to the contact Hamiltonian
systems. Since the latter can be always obtained as the reduction
of a homogeneous symplectic Hamiltonian system, the following question
arises: is the Herglotz variational principle related in some \textit{natural}
way to the homogeneous Lagrangian systems? We show below that the
answer to this question seems to be negative. In other words, we show
that, in some cases, a reduced homogeneous Lagrangian system can not
be described by the Herglotz principle, in the sense that the trajectories
of the former do not coincide with the critical points of the latter.
Also, we show the reciprocal result: a dynamical system whose trajectories
are the critical points of a Herglotz principle, in general, can not
be described as a reduced homogeneous Lagrangian system. These two
statements will be made more precise in the following sections.

\subsection{From homogeneity to Herglotz}

Consider an \textit{action-dependent Lagrangian} $\hat{L}:T\hat{Q}\times\mathbb{R}\rightarrow\mathbb{R}$
on a manifold $\hat{Q}$ and its related Herglotz variational principle
on the interval $\left[0,\tau\right]$ (see Ref. \cite{lainz} for
more details). If we fix an affine connection for the tangent bundle
$\tau_{\hat{Q}}:T\hat{Q}\rightarrow\hat{Q}$, it can be shown that
the critical points of that principle are precisely the solutions
of the \textbf{Herglotz equations}

\begin{equation}
-\frac{D}{Dt}\frac{\partial\hat{L}}{\partial\dot{x}}\left(\dot{x}\left(t\right),y\left(t\right)\right)+\left(\frac{\partial\hat{L}}{\partial y}\,\frac{\partial\hat{L}}{\partial\dot{x}}\right)\left(\dot{x}\left(t\right),y\left(t\right)\right)+\frac{\partial\hat{L}}{\partial x}\left(\dot{x}\left(t\right),y\left(t\right)\right)=0\label{mhe}
\end{equation}
and 
\begin{equation}
\dot{y}\left(t\right)=\hat{L}\left(\dot{x}\left(t\right),y\left(t\right)\right),\label{she}
\end{equation}
for $t\in\left(0,\tau\right)$, where $\partial\hat{L}/\partial\dot{x}$
and $\partial\hat{L}/\partial x$ are defined as in Eqs. \eqref{dfb1}
to \eqref{dfb4} (replacing $Q/\mathbb{R}^{+}$ by $\hat{Q}$, $\ell$
by $\hat{L}$ and ${\mathfrak{g}}$ by ${\mathbb{R}}$). 

A first naive attempt at comparing the Herglotz equations \eqref{mhe}
and \eqref{she} and the scaling-Lagrange-Poincaré equations \eqref{hor}
and \eqref{ver} is to use $\hat{Q}=Q/\mathbb{R}^{+}$ and $\hat{L}=\ell$.
In such a case, it is easy to see both sets of equations not only
are markedly different, but they also have, in general, different
sets of solutions. In fact, for the reduced Lagrangian $\ell:T\hat{Q}\times\mathbb{R}\rightarrow\mathbb{R}$
described in Section \eqref{hlse} (see Eq. \eqref{rll}), where $\hat{Q}$
is the open interval $\left(0,\pi/2\right)$, it can be shown that
its related scaling-Lagrange-Poincaré equations has a solution $\left(x\left(t\right),y\left(t\right)\right)$
with $y\left(t\right)=1$ for all $t$, which is not a solution for
the Herglotz equations related to $\hat{L}=\ell$. On the other hand,
if we consider, for the same $\hat{Q}$, the action-dependent Lagrangian
\begin{equation}
\hat{L}:T\hat{Q}\times\mathbb{R}\rightarrow\mathbb{R}\quad:\quad\left(\dot{x},y\right)\longmapsto\frac{\dot{x}\,y}{e-x}+2\,\dot{x}^{2}-y^{2}/2,\label{adl}
\end{equation}
it follows that its related Herglotz equations coincide exactly with
the scaling-Lagrange-Poincaré equations (for the above mentioned reduced
Lagrangian $\ell$). In particular, the equations \eqref{mhe} and
\eqref{she} for $\hat{L}$ and the equations \eqref{hor} and \eqref{ver}
for $\ell$ have the same set of solutions. So, we can describe the
homogeneous Lagrangian system with reduced Lagrangian \eqref{rll}
in terms of the Herglotz variational principle defined by \eqref{adl}.
It is natural to ask if such a description is possible for every homogeneous
Lagrangian system. In other words, we want to address the following
problem: given a homogeneous Lagrangian function $L:TQ\rightarrow\mathbb{R}$,
is there an action-dependent Lagrangian $\hat{L}:T\hat{Q}\times\mathbb{R}\rightarrow\mathbb{R}$,
for $\hat{Q}=Q/\mathbb{R}^{+}$, such that the solutions of its Herglotz
equations \eqref{mhe} and \eqref{she} coincide with the solutions\footnote{Note that we do not require the involved equations to coincide, but
rather their respective solution sets.} of the scaling-Lagrange-Poincaré equations related to $L$? Below,
we show that, in general, the answer to the question is negative.\bigskip{}

Consider the manifold $Q=\mathbb{R}^{n}-\left\{ 0\right\} $, the
principal action $\psi$ given by \eqref{fsx} and the scaling function
$f$ given by \eqref{fx}. It is easy to see that the diffeomorphism
$\alpha_{f}$, identifying $TQ/{\mathbb{R}}^{+}$ with $T(Q/{\mathbb{R}}^{+})\times\mathbb{R}$,
is given by 
\[
{\alpha}_{f}\left(p\left(\mathbf{q},\dot{\mathbf{q}}\right)\right)=\left(\pi_{*,\mathbf{q}}\left(\dot{\mathbf{q}}\right),\frac{2\,\left\langle \mathbf{q},\dot{\mathbf{q}}\right\rangle }{\left\Vert \mathbf{q}\right\Vert ^{2}}\right)\eqqcolon\left(\dot{\mathbf{x}},y\right),
\]
where $\left\langle \cdot,\cdot\right\rangle $ denotes the Euclidean
metric in $\mathbb{R}^{n}$ and $\left\Vert \cdot\right\Vert $ its
related norm. On the other hand, it is clear that the Lagrangian function
\begin{equation}
L\left(\mathbf{q},\dot{\mathbf{q}}\right)=2k\,\left\langle \mathbf{q},\dot{\mathbf{q}}\right\rangle \label{Ls}
\end{equation}
is homogeneous for every $k\in\mathbb{R}$, and that its related reduced
Lagrangian (see Eq. \eqref{rel}) is 
\[
\ell\left(\dot{\mathbf{x}},y\right)=k\,y.
\]
Then, for this $\ell$, the scaling-Lagrange-Poincaré equations \eqref{hor}
and \eqref{ver} are trivial, i.e. any curve $\left(\mathbf{x}\left(t\right),y\left(t\right)\right)$
is a solution of those equations. Suppose there exists an action-dependent
Lagrangian $\hat{L}:T\left(Q/\mathbb{R}^{+}\right)\times\mathbb{R}\rightarrow\mathbb{R}$
such that any curve is a solution of \eqref{mhe} and \eqref{she}.
Then, given $\mathbf{x}_{0}\in Q/\mathbb{R}^{+}$, $\dot{\mathbf{x}}_{0}\in T_{\mathbf{x}_{0}}\left(Q/\mathbb{R}^{+}\right)$
and $y_{0}\in\mathbb{R}$, any curve of the form
\[
\left(\mathbf{x}\left(t\right),y\left(t\right)\right)=\left(\mathbf{x}\left(t\right),y_{0}\right),\quad t\in\left[0,\tau\right],
\]
with $\mathbf{x}\left(0\right)=\mathbf{x}_{0}$ and $\dot{\mathbf{x}}\left(0\right)=\dot{\mathbf{x}}_{0}$,
must be a solution of \eqref{hor} and \eqref{ver}. This means, according
to \eqref{she}, that 
\[
0=\hat{L}\left(\dot{\mathbf{x}}\left(t\right),y_{0}\right),\quad t\in\left[0,\tau\right].
\]
In particular, for $t=0$, we have that $\hat{L}\left(\dot{\mathbf{x}}_{0},y_{0}\right)=0$.
As a consequence, $\hat{L}$ must be identically zero. But in this
case, a curve of the form 
\[
\left(\mathbf{x}\left(t\right),y\left(t\right)\right)=\left(\mathbf{x}\left(t\right),t\right),\quad t\in\left[0,\tau\right],
\]
can not be a solution of \eqref{she} (for $\hat{L}\equiv0$), which
gives rise to a contradiction. Summing up, there is no action-dependent
Lagrangian $\hat{L}$ whose equations of motion have the same solution
set as the one of the scaling-Lagrange-Poincaré equations corresponding
to the given reduced Lagrangian $\ell$.
\begin{rem}
\label{sing} Note that the homogeneous Lagrangian \eqref{Ls} we
have chosen in the previous argument is a singular one (i.e. the Legendre
transformation is not a local diffeomorphism). On the other hand,
the Lagrangian \eqref{L} (with reduced Lagrangian \eqref{rll}),
for which a Herglotz description does exist, is regular. We do not
know if that singular character is actually crucial or not for the
validity of our result. Anyhow, it would be interesting to characterize
the set of all homogeneous Lagrangians for which a Herglotz description
is indeed possible. 
\end{rem}

\subsection{From Herglotz to homogeneity}

\label{S5.2}

Now, we want to solve a problem which is the converse of the previous
one: given an action-dependent Lagrangian $\hat{L}:T\hat{Q}\times\mathbb{R}\rightarrow\mathbb{R}$,
does there exist a homogeneous Lagrangian function $L:TQ\rightarrow\mathbb{R}$,
for some $Q$ and some principal action $\psi$ for which $\hat{Q}=Q/\mathbb{R}^{+}$,
such that the solutions of its related scaling-Lagrange-Poincaré equations
coincide with the solutions of the Herglotz equations for $\hat{L}$?
Of course, for the action-dependent Lagrangian given by \eqref{adl},
the answer is affirmative, but, as we show below, this is not always
the case.

\bigskip{}

Take $\hat{L}\equiv0$. This means that the solutions of \eqref{mhe}
and \eqref{she} are of the form $\left(x\left(t\right),y_{0}\right)$,
where $x\left(t\right)$ is arbitrary and $y_{0}$ is constant. Suppose
there exists a homogeneous Lagrangian function $L$ as described in
the last paragraph. Then, any curve of the form $\left(x\left(t\right),y_{0}\right)$
is a solution of the scaling-Lagrange-Poincaré equations related to
the reduced Lagrangian $\ell:T\hat{Q}\times\mathbb{R}\rightarrow\mathbb{R}$
of $L$. As a consequence, given $y_{0}$, consider the function (writing
the base point $x$ explicitly from now on) 
\[
\sigma\left(x,\dot{x}\right)=\frac{\partial\ell}{\partial y}\left(x,\dot{x},y_{0}\right),
\]
and given $\left(x_{0},\dot{x}_{0}\right)$ consider a curve $x\left(t\right)$
defined on an interval $I$ such that 
\[
\left(x\left(0\right),\dot{x}\left(0\right)\right)=\left(x_{0},\dot{x}_{0}\right)\quad\textrm{and}\quad\sigma\left(x\left(t\right),\dot{x}\left(t\right)\right)=\sigma\left(x_{0},\dot{x}_{0}\right),\quad\forall t\in I,
\]
i.e. $\sigma$ is constant along the curve $\left(x\left(t\right),\dot{x}\left(t\right)\right)$.
This can be done for all $\left(x_{0},\dot{x}_{0},y_{0}\right)$ living
inside some rectangular open subset $U\times V\subseteq T\hat{Q}\times\mathbb{R}$,
as we show in the Appendix. Then, Eq. \eqref{ver} says for $t=0$
that 
\[
y_{0}\,\frac{\partial\ell}{\partial y}\left(x_{0},\dot{x}_{0},y_{0}\right)=\ell\left(x_{0},\dot{x}_{0},y_{0}\right)
\]
on $U\times V$. Solving the last equation for $\ell$ we have that
\[
\ell\left(x,\dot{x},y\right)=K\left(x,\dot{x}\right)\,y,\quad\forall\left(x,\dot{x}\right)\in U,\;\forall y\in V,
\]
for some function $K$. From now on, we shall assume that $U$ is
connected. Going back to \eqref{ver}, given a curve $x\left(t\right)$
such that $\left(x\left(t\right),\dot{x}\left(t\right)\right)\in U$,
we obtain that 
\[
\frac{d}{dt}K\left(x\left(t\right),\dot{x}\left(t\right)\right)=0.
\]
According to the calculations made in the Appendix, if 
\[
\frac{\partial K}{\partial x}\neq0\quad\textrm{or}\quad\frac{\partial K}{\partial\dot{x}}\neq0
\]
for some point of $U$, then the curve $x\left(t\right)$ must satisfy
an ODE, which contradicts the fact that $x\left(t\right)$ is arbitrary.
So, $K$ must be a constant $k$ on $U$ (since $U$ is connected),
i.e. 
\[
\ell\left(x,\dot{x},y\right)=k\,y
\]
on $U\times V$. As a consequence, Eqs. \eqref{hor} and \eqref{ver}
are trivial (i.e. they do not impose any condition). This means that
any curve $y\left(t\right)$ inside $V$ (and not only the constant
one: $y\left(t\right)=y_{0}$) is part of a solution for \eqref{hor}
and \eqref{ver}, which gives rise to a contradiction. Thus, the answer
to the question stated above is negative.
\begin{rem}
An observation similar to Remark \ref{sing} is also in order here.
\end{rem}

\section{Future work}

\label{S6}

In this paper we have actually focused on \textit{positive} homogeneity
only (see Remark \ref{posh}). In a forthcoming paper (see \cite{FeGrMaPa1})
we shall also consider general homogeneity, i.e. homogeneity w.r.t.
a principal action of the entire multiplicative group of real numbers
${\mathbb{R}}^{\times}={\mathbb{R}}-\{0\}$. In that case, the related
principal bundle is not necessarily trivial, i.e. the existence of
a scaling function can not be ensured. Moreover, we shall work in
the more general context of the \textit{pre-symplectic Hamiltonian
systems} (in its variational formulation). They are given by triples
$\left(M,\theta,F\right)$, where $M$ is a manifold, $\theta$ is
a $1$-form and $F$ a function on $M$, and their trajectories are
the critical points of the functional 
\[
S\left(c\right)=\int_{0}^{\tau}\left(\theta\left(\dot{c}\left(t\right)\right)-F\left(c\left(t\right)\right)\right)\,\mathsf{d}t,
\]
defined on curves $\left\{ c:\left[0,\tau\right]\rightarrow M\right\} $,
w.r.t. infinitesimal variations satisfying 
\[
\delta c\left(0\right),\delta c\left(\tau\right)\in\mathsf{O}_{M}.
\]
The Lagrangian systems constitute a particular case, where $M=TQ$
for some manifold $Q$, 
\[
\theta=\mathbb{F}L^{*}\theta_{Q}\quad\textrm{and}\quad F=\left\langle \mathbb{F}L\left(\cdot\right),\cdot\right\rangle -L,
\]
for some function $L:TQ\rightarrow\mathbb{R}$ (as usual, $\theta_{Q}$
indicates the Liouville $1$-form of $T^{*}Q$). We shall study the
reduction of such systems in the presence of scaling symmetries, i.e.
when $\theta$ and $F$ are homogeneous w.r.t. a principal action
of $\mathbb{R}^{\times}$. We shall also study the relationship of
the resulting reduced systems with the Herglotz variational principle.

Another interesting topic that we intend to discuss (see \cite{FeGrMaPa2})
is the development of discrete variational principles that parallel
those presented in this paper for the continuous homogeneous and reduced
systems. This construction should allow us to relate the critical
points of both discrete variational principles and, consequently,
use the critical points of the discrete reduced system to reconstruct
the critical points of the original (unreduced) system. Loosely speaking,
a discretization of a continuous homogeneous system is a discrete
homogeneous system whose critical points -trajectories- approximate,
in some sense, those of the continuous system. Then, once the reduction
of discrete homogeneous systems is well understood, an interesting
and important question that we also intend to discuss is whether the
two processes of discretization and reduction commute.

\section*{Acknowledgments}

We thank the anonymous referees for their helpful comments, which
improved the clarity of this paper. We are especially grateful to
one of them for providing the expression for the action-dependent
Lagrangian used in the argument in Section 5.

J. Fernández and S. Grillo acknowledge the financial support of SIIP-UNCUYO
under grants 06/80020240100069UN and 06/80020240100066UN. Grillo also
thanks CONICET for its support. Additionally, S. Grillo, J. C. Marrero
and E. Padrón acknowledge funding from the Spanish Ministry of Science
and Innovation under grant PID2022-137909-NB-C22.

\section*{Appendix}

Consider open subsets $D\subseteq\mathbb{R}^{nm}$ and $I\subseteq\mathbb{R}$,
denote by 
\[
\left(x,x^{\left(1\right)},\ldots,x^{\left(m-1\right)}\right),\quad\textrm{with}\quad x=\left(x_{1},...,x_{n}\right)\quad\textrm{and}\quad x^{\left(i\right)}=\left(x_{1}^{\left(i\right)},...,x_{n}^{\left(i\right)}\right),
\]
the elements of $D$, and by $y$ the elements of $I$. Consider also
a smooth function $F:D\times I\rightarrow\mathbb{R}$ and, using the
standard notation, consider the related $m$ order ODE (with $n$
unknowns and parametrized by $y$) 
\[
\frac{\partial F}{\partial x^{\left(m-1\right)}}\cdot x^{\left(m\right)}\left(t\right)+\cdots+\frac{\partial F}{\partial x}\cdot x^{\left(1\right)}\left(t\right)=0,\qquad\left(\star\right)
\]
where 
\[
\frac{\partial F}{\partial x^{\left(i\right)}}=\left(\frac{\partial F}{\partial x_{1}^{\left(i\right)}},\ldots,\frac{\partial F}{\partial x_{n}^{\left(i\right)}}\right)\left(x\left(t\right),x^{\left(1\right)}\left(t\right),\ldots,x^{\left(m-1\right)}\left(t\right),y\right),
\]
\[
x^{\left(i\right)}\left(t\right)=\frac{d^{i}}{dt^{i}}x\left(t\right),
\]
for $i=1,...,m-1$, and $x\left(t\right)=\left(x_{1}\left(t\right),...,x_{n}\left(t\right)\right)$
is the unknown of the ODE. The point $\cdot$ denotes the canonical
inner product in $\mathbb{R}^{n}$. We shall show there exists a point
$\left(a,a^{\left(1\right)},\ldots,a^{\left(m-1\right)},b\right)\in D\times I$
and a ``rectangular'' open neighborhood $U\times V\subseteq D\times I$
of it such that, given $\left(x_{0},x_{0}^{\left(1\right)},\ldots,x_{0}^{\left(m-1\right)},y_{0}\right)\in U\times V$,
there exists a curve $x:\left(-\epsilon,\epsilon\right)\rightarrow\mathbb{R}^{n}$
which solves $\left(\star\right)$ for $y=y_{0}$ and satisfies 
\[
\left(x\left(0\right),x^{\left(1\right)}\left(0\right),\ldots,x^{\left(m-1\right)}\left(0\right)\right)=\left(x_{0},x_{0}^{\left(1\right)},\ldots,x_{0}^{\left(m-1\right)}\right).
\]
Let us show it by induction on $m$. For $m=1,$ we have the ODE 
\[
\frac{\partial F}{\partial x}\cdot x^{\left(1\right)}=0,
\]
i.e. 
\[
\frac{\partial F}{\partial x}\left(x\left(t\right),y\right)\cdot\dot{x}\left(t\right)=0.\qquad\left(\star\star\right)
\]
If $\partial F/\partial x=0$ along all of $D\times I$, then any
curve is a solution of $\left(\star\star\right)$, and the result
easily follows. Otherwise, there must exist a point 
\[
\left(a,b\right)\in D\times I\subseteq\mathbb{R}^{n}\times\mathbb{R}
\]
and a neighborhood of it where $\partial F/\partial x\neq0$. In this
case, consider around $\left(a,b\right)$ a local orthonormal basis
$X_{1},...,X_{n}$ of the trivial vector bundle $\left(D\times I\right)\times\mathbb{R}^{n}$
(w.r.t. the canonical inner product on $\mathbb{R}^{n}$) such that
$X_{1}=\partial F/\partial x$. It is clear that the open neighborhood
where such a basis is defined can be taken rectangular. Let us denote
it $U\times V$. Given $y_{0}\in V$, consider the normal system of
ODEs 
\[
\dot{x}\left(t\right)=X_{2}\left(x\left(t\right),y_{0}\right).
\]
As it is well-known, since $X_{2}\left(\cdot,y_{0}\right)$ is continuous,
for every $x_{0}\in U$ there exists a solution $x:\left(-\epsilon,\epsilon\right)\rightarrow\mathbb{R}^{n}$
of above system such that $x\left(0\right)=x_{0}$. And, since $X_{1}\cdot X_{2}=0$,
such a solution is also a solution of $\left(\star\star\right)$ for
$y=y_{0}$. This implies that the wanted result is true for $m=1$.
Assume now the result is true for $m-1$ and let us prove it for $m$.
If 
\[
\frac{\partial F}{\partial x^{\left(m-1\right)}}=0
\]
along all of $D\times I$, then $F$ does not depend on $x^{\left(m-1\right)}$,
Eq. $\left(\star\right)$ reduces to 
\[
\frac{\partial F}{\partial x^{\left(m-2\right)}}\cdot x^{\left(m-1\right)}+\cdots+\frac{\partial F}{\partial x}\cdot x^{\left(1\right)}=0,
\]
and the wanted result follows from the inductive hypothesis. On the
other hand, if $\partial F/\partial x^{\left(m-1\right)}\neq0$ for
some point $\left(a,a^{\left(1\right)},\ldots,a^{\left(m-1\right)},b\right)\in D\times I$,
then it is different from $0$ around that point, and we can consider
as above a local orthonormal basis $X_{1},...,X_{n}$ of $\left(D\times I\right)\times\mathbb{R}^{n}$
with 
\[
X_{1}=\frac{\partial F}{\partial x^{\left(m-1\right)}}.
\]
Let us denote again $U\times V$ the open subset where such a basis
is defined. If we call 
\[
G\coloneqq\frac{\partial F}{\partial x^{\left(m-2\right)}}\cdot x^{\left(m-1\right)}+\cdots+\frac{\partial F}{\partial x}\cdot x^{\left(1\right)}
\]
and 
\[
H\coloneqq\frac{\partial F}{\partial x^{\left(m-1\right)}}\cdot\frac{\partial F}{\partial x^{\left(m-1\right)}}=X_{1}\cdot X_{1},
\]
then Eq. $\left(\star\right)$ translates to 
\[
\frac{\partial F}{\partial x^{\left(m-1\right)}}\cdot x^{\left(m\right)}+\frac{\partial F}{\partial x^{\left(m-1\right)}}\cdot\left(\frac{\partial F}{\partial x^{\left(m-1\right)}}\,\frac{G}{H}\right)=0,
\]
or equivalently 
\[
\frac{\partial F}{\partial x^{\left(m-1\right)}}\cdot\left(x^{\left(m\right)}+\frac{\partial F}{\partial x^{\left(m-1\right)}}\,\frac{G}{H}\right)=0.
\]
So, following similar ideas as above, we can consider the normal system
of ODEs 
\[
x^{\left(m\right)}=X_{2}-\frac{\partial F}{\partial x^{\left(m-1\right)}}\,\frac{G}{H},
\]
for some parameter $y=y_{0}\in V$, whose solutions are also solutions
of $\left(\star\right)$ (for $y=y_{0}$) and they exist for any initial
condition $\left(x_{0},x_{0}^{\left(1\right)},\ldots,x_{0}^{\left(m-1\right)}\right)\in U$.
In this way, we finally arrive at the wanted result for arbitrary
$m$.

\bigskip{}

For $m=2$ we are saying that, given a smooth function $F\left(x,x^{\left(1\right)},y\right)$
(on a domain as described above), there exist a point $\left(a,a^{\left(1\right)},b\right)$
and a rectangular open neighborhood $U\times V$ of it such that,
for every $\left(x_{0},x_{0}^{\left(1\right)},y_{0}\right)\in U\times V$,
there exists a curve $x\left(t\right)$ such that 
\[
\frac{d}{dt}\left(F\left(x\left(t\right),\dot{x}\left(t\right),y_{0}\right)\right)=0\quad\textrm{and}\quad\left(x\left(0\right),\dot{x}\left(0\right)\right)=\left(x_{0},x_{0}^{\left(1\right)}\right).
\]

\end{document}